\documentclass[12pt]{amsart}

\usepackage{euscript,amssymb,xr}

\let\oldmarginpar\marginpar
\renewcommand\marginpar[1]
{\oldmarginpar{\tiny\bf \begin{flushleft} #1 \end{flushleft}}}

\input xy
\xyoption{all}

\newcommand{\la}{\langle}
\newcommand{\ra}{\rangle}

\newtheorem{theorem}{\bf Theorem}[section]
\newtheorem{lemma}[theorem]{\bf Lemma}

\newtheorem{remark}[theorem]{\bf Remark}
\newtheorem{prop}[theorem]{\bf Proposition}
\newtheorem{corollary}[theorem]{\bf Corollary}

\newtheorem{question}[theorem]{\bf Question}

\newcommand{\CC}{{\Bbb C}}

\newcommand{\NN}{{\Bbb N}}

\newcommand{\QQ}{{\Bbb Q}}
\newcommand{\RR}{{\Bbb R}}

\newcommand{\ZZ}{{\Bbb Z}}

\newcommand{\plie}{{\frak p}}

\newcommand{\ggreat}{>\kern-.7ex>}
\newcommand{\ssmall}{<\kern-.7ex<}
\newcommand{\qu}{/\kern-.7ex/}
\newcommand{\exh}{\to\kern-1.8ex\to}

\newcommand{\dD}{{\EuScript{D}}}

\newcommand{\fF}{{\EuScript{F}}}
\newcommand{\gG}{{\EuScript{G}}}

\newcommand{\mM}{{\EuScript{M}}}
\newcommand{\oO}{{\EuScript{O}}}

\newcommand{\pP}{{\EuScript{P}}}
\newcommand{\sS}{{\EuScript{S}}}

\newcommand{\GL}{\operatorname{GL}}

\newcommand{\Aff}{\operatorname{Aff}}
\newcommand{\Aut}{\operatorname{Aut}}

\newcommand{\Bir}{\operatorname{Bir}}

\newcommand{\Coker}{\operatorname{Coker}}

\newcommand{\Hom}{\operatorname{Hom}}
\newcommand{\Homeo}{\operatorname{Homeo}}
\newcommand{\Id}{\operatorname{Id}}

\newcommand{\Ker}{\operatorname{Ker}}

\newcommand{\SL}{\operatorname{SL}}

\newcommand{\rk}{\operatorname{rk}}

\newcommand{\Spec}{\operatorname{Spec}}
\newcommand{\Stab}{\operatorname{Stab}}

\newcommand{\wt}{\widetilde}

\newcommand{\ov}{\overline}

\newcommand{\discsym}{\operatorname{disc-sym}}
\newcommand{\absym}{\operatorname{tor-sym}}
\newcommand{\torsym}{\operatorname{tor-sym}}

\title[Discrete degree of symmetry of manifolds]
{Discrete degree of symmetry of manifolds}
\email{ignasi.mundet@ub.edu}

\date{\today}


\thanks{This research was partially supported by the grant
PID2019-104047GB-I00 from the Spanish Ministeri de Ci\`encia i Innovaci\'o.}

\author[I. Mundet i Riera]{Ignasi Mundet i Riera}
\address{ignasi.mundet@ub.edu \\ Departament de Matem\`atiques i Inform\`atica, Universitat de Barcelona, Spain}

\usepackage[foot]{amsaddr}

\begin{document}

\maketitle

%\begin{abstract}
%REWRITE!!!
%Let $X$ be a closed, connected and oriented topological $n$-dimensional manifold.
%Suppose that:
%\begin{enumerate}
%\item $X$ supports
%an effective action of $(\ZZ/r)^n$ for arbitrarily large values of $r$,
%\item the fundamental group $\pi_1(X)$
%is virtually solvable, and
%\item there exists a map $X\to T^n$ of nonzero degree.
%\end{enumerate}
%We then prove that $X$ is homeomorphic to $T^n$.
%Assuming only (1) and (3) we prove that the cohomology ring
%$H^*(X;\ZZ)$ is isomorphic to $H^*(T^n;\ZZ)$.
%We also prove that any closed connected and oriented $n$-dimensional
%manifold $X$ admitting a map to $T^n$ of nonzero degree has Jordan homeomorphism group, and
%we study for such manifolds the maximal number $k$ such that $X$
%admits effective actions of $(\ZZ/r)^k$ for arbitrarily large $r$. Finally, we prove that if $X$
%is a compact and connected Kaehler manifold of real dimension $n$ and $X$ supports effective holomorphic actions
%of $(\ZZ/r)^m$ for arbitrarily large $r$, then $m\leq n$, and if $m=n$ then $X$ is biholomorphic to
%a complex torus.
%\end{abstract}

\begin{abstract}
We define the discrete degree of symmetry $\discsym(X)$
of a closed $n$-manifold $X$ as the biggest $m\geq 0$
such that $X$ supports an effective action of $(\ZZ/r)^m$ for arbitrarily big values of $r$. We prove
that if $X$ is connected then $\discsym(X)\leq 3n/2$. We propose the question of
whether for every closed connected $n$-manifold $X$ the inequality $\discsym(X)\leq n$
holds true, and whether the only closed connected $n$-manifold $X$ for which $\discsym(X)=n$ is the torus $T^n$.
We prove partial results providing evidence for an affirmative answer to this question.
\end{abstract}

\section{Introduction}

Let $X$ be a closed topological manifold
and let $\mu(X)$ be the set of natural numbers
$m$ for which there exists an effective action\footnote{In this paper all group actions on manifolds are implicitly
assumed to be continuous.} of $(\ZZ/r)^m$ on $X$
arbitrarily large values of $r$. In other words, $m\in\mu(X)$ means that there
exists a sequence of integers $r_i\to\infty$ and an effective action of $(\ZZ/r_i)^m$
on $X$ for each $i$. By the theorem of Mann and Su
\cite{MS} the set $\mu(X)$ is finite. Define the {\it discrete degree of symmetry} of $X$ to be
the number
$$\discsym(X):=\max (\{0\}\cup\mu(X)).$$

The discrete degree of symmetry can be equivalently defined taking into
account all finite abelian groups that act effectively on a given manifold, and not
only those of the form $(\ZZ/r)^m$. Indeed, we will prove in
Lemma \ref{lemma:discsym-rank-groups} that
$\discsym(X)$ %the discrete degree of symmetry of $X$
coincides with the
smallest nonnegative integer $k$ for which there exists a constant $C$ such that
any finite abelian group $A$ acting effectively on $X$ has a subgroup $A'$ satisfying
$[A:A']\leq C$ and which can be generated by $k$ or fewer elements.

%In this paper we prove several results on the discrete degree of symmetry of closed connected
%manifolds. We will see that, while in general $\aA(X)$ is very difficult
%to describe, $\discsym(X)$ is, at least in some cases, much more amenable to computations.
%Furthermore, if $X$ is connected then $\discsym(X)$ is bounded by a linear function of
%$\dim X$ (see Theorem ******; we actually conjecture that
%$\discsym(X)\leq \dim X$), while one cannot bound the rank of the elements in $\aA(X)$
%in terms of the dimension. For example, for any integer $d\geq 2$ and any finite abelian group
%$G$ there exist a closed connected manifold $X$ of dimension $d$ such that $[G]\in\aA(X)$.

%In this paper we prove that if $X$ is a rationally hypertoral $n$-dimensional manifold (see below) then
%$\discsym(X)\leq n$, and if $\discsym(X)=n$ then $H^*(X;\ZZ)\simeq H^*(T^n;\ZZ)$ as rings. The latter
%can be strengthened if $\pi_1(X)$ is virtually solvable: in this case, $\discsym(X)=n$ implies that
%$X\cong T^n$.

%Based on the previous results, we ask whether the inequality $\discsym(X)\leq\dim X$
%holds for every closed connected manifold $X$, and whether tori are the only closed connected manifolds
%for which $\discsym(X)=\dim X$. Besides our results on rationally hypertoral manifolds, we provide
%several additional arguments, some of them heuristic and others rigorous, in favor of a positive answer to these expectations.

One can regard the discrete degree of symmetry as an analogue for finite groups of the
{\it (topological) torus-degree of symmetry \cite[p. 132]{Hsiang}}, which has been extensively studied in
the literature (see e.g. \cite{BS,Hsiang,KK,LeeRaymond2,Schultz-2}). For a closed manifold $X$,
this is defined as the maximal $n$ for which the torus $T^n$ admits a continuous effective action on $X$ (by convention $T^0=\{1\}$).
We denote it by $\torsym(X)$. The torus degree of symmetry
is also called in some references the {\it toral rank}, see e.g. \cite[\S 11.8.1]{LeeRaymond2},
although in some other references the expression {\it toral rank} is reserved for free actions.

It is well known that if $X$ is a connected topological $n$-manifold then $\torsym(X)\leq n$,
with equality if and only if $X$ is homeomorphic to $T^n$ (see Subsection \ref{ss:proof-thm:Kaehler}). Since $(\ZZ/r)^m$ is isomorphic to
the $r$-torsion $T^m[r]<T^m$ and, for any sequence $r_i\to\infty$, the union $\bigcup_i T^m[r_i]$ is dense
in $T^m$, it seems natural to expect that a closed connected manifold $X$ satisfying $\discsym(X)=m$
should look somehow as if it supported an effective action of $T^m$. This heuristic can be turned
into an actual theorem in some situations (see e.g. Theorems \ref{thm:Kaehler} and \ref{thm:Lie-groups}), but it has its limitations:
while trivially $\discsym(X)\geq\torsym(X)$ for any closed manifold
$X$, there are examples of closed connected manifolds manifolds for which the inequality is strict, as we will
see below (see Theorem \ref{thm:CWY}).
Nevertheless, it may still be true that the bound $\torsym\leq\dim$ is also satisfied by $\discsym$.
We thus ask the following.

\begin{question}
\label{quest:bound-disc-sym}
Is the inequality $\discsym(X)\leq \dim X$ true for every closed connected manifold $X$?
If a closed connected manifold $X$ satisfies $\discsym(X)=\dim X$, is $X$ homeomorphic to a torus?
\end{question}

If one removes the condition that $X$ is connected then there is no hope to bound
$\discsym(X)$ by a function on the dimension of $X$. For example, the disjoint
union of $k$ copies of the circle supports an effective action of $T^k$, where
the action on the $j$-th circle is given by the projection to the
$j$-th factor $T^k=(S^1)^k\to S^1$.

If one considers only free actions of $(\ZZ/r)^m$ on connected manifolds, then the first part of Question \ref{quest:bound-disc-sym} follows from a theorem of Baumgartner and Carlsson
\cite[Theorem 1.4.14]{AP} (see Theorem \ref{thm:Carlsson-Baumgartner}) and a lemma of
Minkowski (see Theorem \ref{thm:Minkowski}).

In this paper we prove several results partially answering Question \ref{quest:bound-disc-sym}
in the affirmative, as well as other results related to the discrete degree of symmetry.
More evidence in favour of Question \ref{quest:bound-disc-sym} is provided
in \cite{Mundet2022}.

%We trivially have $\discsym(X)\geq \torsym(X)$, and in this paper we prove that
%for the manifolds constructed by Cappell, Weinberger and Yan \cite{CWY}
%the previous inequality is strict.

One may consider variants of the discrete degree of symmetry by considering only actions
of $(\ZZ/p)^m$ for $p$ prime, by considering only free actions, or by combining both
restrictions. Some of these variations have been studied in the literature, see e.g.
\cite{Hanke}.

An interesting class of closed manifolds with nonzero discrete degree of symmetry are
closed strongly regular self-covering manifolds \cite{vL,QSW}, i.e.,
closed manifolds which are homeomorphic to a nontrivial (necessarily finite) regular covering of themselves and which satisfy the additional
property that each iterated covering is regular.
%This observation is one of the ingredients in
%the proof of Theorem \ref{thm:CWY}.
%There is a rich literature on self-covering manifolds,
%see e.g. \cite{QSW,vL} and the references therein.
There are interesting relations
between some of the results in the present paper about rationally hypertoral manifolds
and some results in \cite{QSW} (see below).

\subsection{A bound on the discrete degree of symmetry}
Our first result is the following.

\begin{theorem}
\label{thm:weak-bound-disc-sym}
For any closed connected $n$-manifold $X$ we have $\discsym(X)\leq 3n/2$.
\end{theorem}

Quantitatively, the previous theorem stays far from the bound suggested by
Question \ref{quest:bound-disc-sym}, but it shows a qualitative difference
between the discrete degree of symmetry and the rank of individual groups
acting effectively on a given manifold.
%Indeed, assuming a
%group of the form $(\ZZ/r)^m$ acts effectively on an $n$-dimensional closed
%connected manifold, there is no way to bound $m$ above by a function of $n$.
The theorem of Mann and Su \cite{MS} mentioned earlier implies that if $X$ is
a closed connected manifold and $(\ZZ/r)^m$ acts effectively on $X$ then
$m$ is bounded by a function involving the Betti numbers of $X$. Although
the bound given in \cite{MS} could possibly be improved, there is no hope to
replace it by a constand depending only on the dimension of $X$,
because any finite group acts freely (hence effectively)
on some closed connected surface.

\subsection{Discrete degree of symmetry of rationally hypertoral manifolds}
A closed, connected and oriented $n$-dimensional manifold $X$ is said to be {\it rationally hypertoral}
if it admits a continuous map $\phi:X\to T^n$ of nonzero degree. If the map $\phi$ can be chosen
of degree $\pm 1$, then $X$ is said called {\it hypertoral} in \cite{Schultz-1,Schultz-2}.
Equivalently, $X$ is rationally hypertoral if it admits classes $\alpha_1,\dots,\alpha_n\in H^1(X;\ZZ)$
such that $\alpha_1\smile\dots\smile\alpha_n\neq 0$, because $T^n=(S^1)^n$
and $S^1$ is an Eilenberg-MacLane space $K(\ZZ,1)$. Similarly, $X$ is hypertoral
if it admits classes $\alpha_1,\dots,\alpha_n\in H^1(X;\ZZ)$ such that
$\alpha_1\smile\dots\smile\alpha_n$ is a generator of $H^n(X;\ZZ)$. For example, if $X$ is any
closed, connected and orientable $n$-manifold then the connected sum $T^n\sharp X$ is a hypertoral
manifolds. Similarly, if $X$ is a closed complex submanifold of $\CC^n/\Lambda$, where $\Lambda$ is a lattice, then $X$ is rationally hypertoral (see \cite[p. 243]{Yau}). See Theorem \ref{thm:non-hypertoral} for examples of non hypertoral rationally
hypertoral manifolds.

\begin{theorem}
\label{thm:discsym-rationally-hypertoral}
Let $X$ be a rationally hypertoral $n$-manifold. We have:
\begin{enumerate}
\item $\discsym(X)\leq n$;
\item if $\discsym(X)=n$ then the universal abelian cover of $X$ is acyclic and there is an isomorphism of rings $H^*(X;\ZZ)\simeq H^*(T^n;\ZZ)$.
\end{enumerate}
\end{theorem}

We recall some standard terminology for the reader's convenience.
A covering space $Y\to X$ is abelian if it is regular and its group of deck transformations
is abelian. Choose a base point in $X$, let $\pi=\pi_1(X)$ and let $X'$ be the universal cover
space of $X$. A connected cover $f:Y\to X$ is abelian if $f_*\pi_1(Y)$ contains $[\pi,\pi]$.
The universal abelian cover of $X$ can be identifed with $X'/[\pi,\pi]$.
A connected abelian cover $Y\to X$ is isomorphic to the universal abelian cover of $X$
if and only if $H_1(Y)=0$.

%Statement (2) in Theorem \ref{thm:discsym-rationally-hypertoral}
%implies that if a rationally hypertoral manifold
%$X$ satisfies $\discsym(X)=\dim X$ then $X$ is actually hypertoral.

The converse of (2) in Theorem \ref{thm:discsym-rationally-hypertoral}
is not true. For example, if $X$ is the connected sum of $T^3$ and Poincar\'e's sphere,
then $H^*(X;\ZZ)\simeq H^*(T^3;\ZZ)$ and the universal abelian cover of $X$ is acyclic.
However, $\discsym(X)=0$. Indeed, $\discsym(X)>0$ would imply that $X$ supports a circle action,
by \cite{Pardon2019} and the geometrization of $3$-manifolds (see e.g. the arguments
in \cite[\S 2]{Zimmermann2014}). But $X$ does not support any circle
action, by \cite[Theorem 5.1]{Schultz-2}.

Statement (1) in Theorem \ref{thm:discsym-rationally-hypertoral} follows from a somewhat routine
extension to continuous actions of the construction described in \cite[\S 2.1]{M1} and
\cite[\S 8.1]{MundetSaez}. The proof of (2)
is based on the following result on commutative algebra (see Corollary
\ref{cor:finitely-generated-plus}), which is perhaps of independent interest.

\begin{theorem}
\label{thm:finite-generation}
Let $M$ be a finitely generated module over $A:=\ZZ[t_1^{\pm 1},\dots,t_n^{\pm 1}]$.
Suppose that for every $1\leq i\leq n$ there exists a nonzero integer $d_i$,
a sequence of integers $(r_{i,j})_j$ satisfying $r_{i,j}\to\infty$ as $j\to\infty$,
and $A$-module automorphisms $w_{i,j}:M\to M$ such that $w_{i,j}^{r_{i,j}}$ coincides with multiplication
by $t_i^{d_i}$. Then $M$ is finitely generated as a $\ZZ$-module.
\end{theorem}

A similar theorem was proved independently, with a rather different proof, by
L. Qin, Y. Su and B. Wang in \cite[Theorem G]{QSW} (both the first version of \cite{QSW} and
that of the present paper
were posted almost simultaneously in the arxiv). One can actually derive Theorem \ref{thm:finite-generation} 
from \cite[Theorem G]{QSW}\footnote{I thank L. Qin, Y. Su and B. Wang for explaining this to me.}. Since this derivation
is nontrivial and our proof of Theorem \ref{thm:finite-generation} is elementary,
not longer, and different from that in \cite[Theorem G]{QSW}, it is perhaps
worthwhile to keep it in this revised version.

Theorem \ref{thm:finite-generation} is also used in the proof of the following result.

\begin{theorem}
\label{thm:bounding-discsym}
Let $X$ be a rationally hypertoral manifold.
Suppose that $X$ is homeomorphic to $(Y\sharp Y')\times Z$,
where $Y,Y',Z$ are closed connected topological manifolds satisfying
$\dim Y=\dim Y'>1$. If neither $Y$ nor $Y'$ are
integral homology spheres, then $\discsym(X)\leq \dim Z$.
\end{theorem}

An analogue of the previous theorem for $\torsym$ instead of $\discsym$,
and for the case where both $Y$ and $Z$ are tori, was proved by Schultz
in \cite[Theorem 5.1]{Schultz-2} (Schulz assumes that $Y'$ is not a homotopy
sphere, which is slightly weaker than our assumption).
%More precisely, Schultz proves in
%[op. cit.] that if $Z$ is a closed connected $l$-dimensional topological manifold
%then $\absym(T^{k}\times (T^{l}\sharp Z))\leq k$
%unless $Z$ is a homotopy sphere.
%The previous theorem
%is a strengthening of Schultz's result, except that our hypothesis that neither $Z$ nor $Z'$ are
%integral homology spheres is stronger than the assumption that neither of them is a homotopy %sphere.

Using the topological rigidity of tori, we deduce from (2) in Theorem
\ref{thm:discsym-rationally-hypertoral} the following.

\begin{corollary}
\label{cor:homeomorphic-torus}
Let $X$ be a rationally hypertoral $n$-manifold such that $\pi_1(X)$ is virtually solvable.
If $\discsym(X)=n$ then $X$ is homeomorphic to $T^n$.
\end{corollary}

Topological rigidity of tori is the statement that if $X$ is a
closed connected manifold then any homotopy equivalence $X\to T^n$ is homotopic to a homeomorphism.
If $n\leq 2$ this is a consequence of the classification of compact connected manifolds
of dimensions at most $2$. It was proved for $n\geq 5$ by Hsiang and Wall \cite{HW},
for $n=4$ by Freedman \cite[\S 11.5]{FQ} (see also \cite{BKKPR}), and in dimension
$n=3$ it is a consequence of Thurston's geometrisation conjecture proved by Perelman
(see \cite{BBMBP,MorganTian} for proofs of the geometrisation conjecture, and \cite[\S 5]{KL} for the proof
that geometrisation implies topological rigidity of $T^3$).
Topological rigidity of tori is
a particular case of Borel's conjecture (see e.g. \cite[\S 3]{L} for a survey, \cite{Farrell, FJ} for
the case of Riemannian manifolds with nonpositive curvature, and also the recent textbook \cite{CW}).

The following result complements Corollary \ref{cor:homeomorphic-torus} in the smooth category.

\begin{theorem}
\label{thm:main-smooth} Let $n\neq 4$ be a natural number.
Let $X$ be a smooth manifold homeomorphic to $T^n$. Then:
\begin{enumerate}
\item $X$ supports effective smooth actions of $(\ZZ/r)^n$ for
arbitrarily large values of $r$;
\item there exists a number $\delta(n)$ (depending on $n$ but not on $X$)
such that if $X$ supports an effective smooth action of $(\ZZ/m\delta(n))^n$ for
some nonzero integer $m$ then $X$ is diffeomorphic to $T^n$.
\end{enumerate}
\end{theorem}

Statement (2) above is related to many results
in the literature showing that homeomorphic but not diffeomorphic manifolds
need not support smooth effective actions of the same finite or compact groups
(see e.g. \cite{HH} and the references therein for systematic results on the case of
the spheres and \cite{BT,BKKT} for analogous questions on tori).

Combining Corollary \ref{cor:homeomorphic-torus} and Theorem \ref{thm:main-smooth} we obtain:

\begin{corollary}
Let $n\neq 4$ be a natural number. Let $X$ be a closed, connected and oriented $n$-dimensional
rationally hypertoral smooth manifold. If $X$
supports an effective action of $(\ZZ/r)^n$ for every natural number $r$
and $\pi_1(X)$ is virtually solvable then $X$ is diffeomorphic to $T^n$.
\end{corollary}

\subsection{Holomorphic discrete degree of symmetry of compact Kaehler manifolds}
The following result answers affirmatively the analogue of Question \ref{quest:bound-disc-sym}
for holomorphic actions on Kaehler manifolds.

\begin{theorem}
\label{thm:Kaehler}
Let $X$ be a compact connected Kaehler manifold of real dimension $n$. Suppose that, for some
natural number $m$, $X$ supports an effective holomorphic action of $(\ZZ/r)^m$
for arbitrarily large values of $r$. Then $X$ supports an effective holomorphic action of
$T^m$. Furthermore, $m\leq n$, and if $m=n$ then $X$ is biholomorphic to a complex torus.
\end{theorem}

We will prove Theorem \ref{thm:Kaehler} using a result of Fujiki \cite{Fuj} on automorphism groups
of compact Kaehler manifolds and the following result on Lie groups.

\begin{theorem}
\label{thm:Lie-groups}
Let $G$ be a (finite dimensional) Lie group with finitely many connected components.
For every natural number $n$ the following properties are equivalent:
\begin{enumerate}
\item $G$ has a Lie subgroup isomorphic to $T^n$,
\item $G$ has a subgroup isomorphic to $(\ZZ/r)^n$ for arbitrary large integers $r$.
\end{enumerate}
\end{theorem}

%The previous theorem does not extend to the infinite dimensional setting.
%Indeed, if we take $G=\Diff (X)$ for some manifold $X$ then the equivalence
%of statements (1) and (2) of the theorem applied to this $G$ is the same
%thing as the equality of the smooth analogues of $\discsym(X)$ and $\absym(X)$,
%which as we explained previously is false in general.

Theorem \ref{thm:Kaehler} seems to be new even for smooth projective varieties
over the complex numbers.
%It is natural to wonder whether the result is also true
%for varieties defined over other fields of characteristic zero.
One can also ask the analogous question for birational transformation groups.
Namely, if $X$ is an $n$-dimensional variety defined over the complex numbers
(or more generally a field of characteristic zero)
and the birational transformation group $\Bir(X)$ contains subgroups isomorphic to
$(\ZZ/r)^m$ for arbitrarily large values of $r$, does it follow that $m\leq 2n$?
If $m=2n$, does it follow that $X$ is birational to an abelian variety?
\footnote{After this paper was finished, these questions have been answered in the
affirmative by A. Golota, see \cite{Gol}.}

A partial result on the first question, due to Prokhorov and Shramov, appears in
\cite[Theorem 1.10]{PS}. An analogue of the second question for rationally connected
varieties has been recently proved by Xu \cite[Theorem 1.3]{Xu}: namely,
if $X$ is a rationally connected $n$-dimensional variety and $\Bir(X)$ contains
subgroups isomorphic to $(\ZZ/p)^n$ for sufficiently big primes $p$ then $X$
is rational.

\subsection{Discrete degree of symmetry and torus-degree of symmmetry}
Corollary \ref{cor:homeomorphic-torus}
proves that, at least in some particular situations, if
$\discsym(X)=\dim X$ then $\absym(X)=\discsym(X)$.
%Maybe the equality $\absym(X)=\discsym(X)$ holds true for {\it every}
%closed connected manifold $X$ satisfying $\discsym(X)=\dim X$, but in any case
But there are examples of closed manifolds for which $\absym<\discsym$,
as proved by a construction due to Cappell, Weinberger and Yan.

\begin{theorem}
\label{thm:CWY}
Let $X$ be any of the manifolds $T(h)\times H$ constructed in \cite[\S 2]{CWY}.
We have $\discsym(X)\geq 1$ and $\absym(X)=0$.
\end{theorem}

%This theorem is not new.
The equality $\absym(X)=0$ is the main result in \cite{CWY}.
The inequality $\discsym(X)\geq 1$ follows from the existence of regular self coverings
$X\to X$ of degree $d$ for every odd natural number $d$.
%$3$, which implies by iteration that $X$ supports a free action of $\ZZ/3^r$
%for every natural number $r$.
The existence of such self coverings (for $d=3$, and hence also for $d=3^k$)
is stated without proof in \cite[Remark 1.3]{vL}, and our contribution in this paper
is to provide a proof (actually, for any odd $d$)
in Section \ref{s:CWY}. A key ingredient in the proof is the topological rigidity of tori.

There are obvious
analogues of the invariants $\discsym$ and $\absym$ for locally linear actions and for smooth
actions on smooth manifolds, and in neither of these categories does one have the equality
$\discsym=\absym$ in general. For locally linear actions there are counterexamples in dimension
$4$, by the work of Edmonds \cite{Edmonds} and Huck \cite{Huck}.
In the smooth category one may take $X=T^n\sharp\Sigma$, where $\Sigma$ is an exotic $n$-sphere.
Then $X$ is homeomorphic to $T^n$, so
$\discsym(X)=n$ by Theorem \ref{thm:main-smooth}, but $\absym(X)=0$ by
the main result in \cite{AssadiBurghelea}. In contrast, for holomorphic actions on compact Kaehler manifolds
one does have $\discsym=\absym$ in general, as proved by Theorem \ref{thm:Kaehler} below.

\subsection{Discrete degree of symmetry and covering spaces}

In the proof of statement (2) in Theorem \ref{thm:discsym-rationally-hypertoral} we reduce
the general case to that in which $\pi_1$ is solvable using the following
result.

\begin{theorem}
\label{thm:discsym-covering}
Let $X$ be a closed connected manifold and let $X'\to X$ be a finite covering.
We have $\discsym(X')\geq\discsym(X)$.
\end{theorem}

The inequality in Theorem \ref{thm:discsym-covering}
can be strict in some cases, as the following theorem proves.
Here and in the rest of the paper we identify $T^n$ with $(\RR/\ZZ)^n$,
so we use additive notation for the group structure on $T^n$.

\begin{theorem}
\label{thm:strict-inequality}
Fix natural numbers $k,n$ satisfying $1\leq k\leq n-1$. Consider the free involution
$\sigma:T^n\to T^n$ defined by
$\sigma(x_1,\dots,x_n)=(x_1+1/2,\dots,x_{k}+1/2,-x_{k+1},\dots,-x_n)$.
Let $X'=T^n$ and let $X=T^n/\sigma$. The natural projection $\rho:X'\to X$ is a covering map
and we have $\discsym X'=n$ and $\discsym X=k$.
\end{theorem}

For example, setting $n=2$ and $k=1$ the manifold $X$ is the Klein bottle and $X'$, the $2$-torus,
is the orientation $2$-cover of $X$.

\subsection{Jordan property and bounds on stabilizers for hypertoral manifolds}

The tools used to prove (1) in Theorem \ref{thm:discsym-rationally-hypertoral} lead to other results
of finite group actions on rationally hypertoral manifolds.

If $X$ is a set supporting an action of a group $G$ we denote
$\Stab(X,G)=\{G_x\mid x\in X\}$ the set of stabilizers of points in $X$. The following
result gives a positive answer to \cite[Question 1.8]{CMPS} for rationally hypertoral manifolds.

\begin{theorem}
\label{thm:CMPS-hypertoral}
Let $X$ be a rationally hypertoral manifold. There exists a constant $C$ such that every
finite group $G$ acting on $X$ has a subgroup $G_0\leq G$ satisfying $[G:G_0]\leq C$
and $|\Stab(X,G)|\leq C$.
\end{theorem}

Recall that a group $\gG$ is
said to be {\it Jordan} if there exists a constant $C$ such that every finite subgroup
$G\leq\gG$ has an abelian subgroup $A\leq G$ satisfying $[G:A]\leq C$.
The following result extends to the topological category the first part of
\cite[Theorem 1.4]{M1}, and it also partially extends \cite[Corollary 1.7]{Ye}.

\begin{theorem}
\label{thm:Jordan}
Let $X$ be a rationally hypertoral manifold. Then the homeomorphism group of $X$ is Jordan.
\end{theorem}

\subsection{Contents}
Section \ref{s:finite-abelian-groups} contains a few elementary results on finite groups
that will be used repeatedly in the paper. In Section \ref{ss:proof-thm:weak-bound-disc-sym}
we prove Theorem \ref{thm:weak-bound-disc-sym}. In Section \ref{s:equivariant-maps-torus}
we prove a result relating finite group actions and maps to tori of nonzero degree.
%that if $X$ is a closed connected oriented $n$-manifold and $\phi:X\to T^n$
%has nonzero degree, then for any finite group action on $X$ which is trivial on $H^1(X;\ZZ)$
%we can replace $\phi$ by a homotopic map which is equivariant with respect to an action on
%$T^n$ by translations. Furthermore, the size of the kernel of the later action is bounded
%above by the degree of $\phi$.
This result is used in Section \ref{s:consequences-thm:linearising-action} to prove the first part of Theorem \ref{thm:discsym-rationally-hypertoral}, and also to prove Theorems \ref{thm:CMPS-hypertoral}
and \ref{thm:Jordan}. Section \ref{s:commutative-algebra} contains the proof of Theorem
\ref{thm:finite-generation} (which is Corollary \ref{cor:finitely-generated-plus}).
In Section \ref{s:coverings} we prove Theorems \ref{thm:discsym-covering} and \ref{thm:strict-inequality} on covering spaces.
In Section \ref{s:cohomology-finitely-generated} we prove that the homology of some
abelian covers is finitely generated as a module over the group ring of the group of deck transformations.
After these preliminaries, in Section \ref{s:proof-thm:main-thm} we prove Theorem \ref{thm:discsym-rationally-hypertoral} and Corollary \ref{cor:homeomorphic-torus},
and in Section \ref{s:proof-thm:bounding-discsym} we prove Theorem \ref{thm:bounding-discsym}.
Section \ref{s:proof-thm:main-smooth} contains the proof of Theorem \ref{thm:main-smooth}. In Section \ref{s:holomorphic}
we study the holomorphic analogues of the discrete degree of symmetry of closed Kaehler
manifolds, and we prove Theorems \ref{thm:Lie-groups} and \ref{thm:Kaehler}.
Finally, in Section \ref{s:CWY}
we prove Theorem \ref{thm:CWY} on the existence of closed manifold whose toral rank is strictly
bigger than its discrete degree of symmetry.

\subsection{Notation} For every finite set $S$ we denote by $|S|$ the cardinality of $S$.
We use additive notation for abelian groups.
For any natural numbers $a,b$ we denote for convenience
$$\Gamma_{a,b}:=(\ZZ/a)^b.$$

\subsection{Acknowledgements} I wish to thank Jordi Daura for pointing out to me reference
\cite{vL}, which led to Theorem \ref{thm:CWY}. Thanks also to Bandi Szab\'o and Costya Shramov
for useful comments.

\section{Some lemmas on finite abelian groups}
\label{s:finite-abelian-groups}

\newcommand{\Obj}{\operatorname{Obj}}

\begin{lemma}
\label{lemma:subgroup-isomorphic-to-Gamma-a-b}
Let $a,b,C$ be natural numbers and suppose that $\Gamma'$ is a subgroup of
$\Gamma_{a,b}$ satisfying $[\Gamma_{a,b}:\Gamma']\leq C$. There exists
a subgroup $\Gamma''\leq\Gamma'$ which is isomorphic to $\Gamma_{a',b}$ for
some natural number $a'$ dividing $a$ and satisfying $C!a'\geq a$.
\end{lemma}
\begin{proof}
Let $d=GCD(a,C!)$ and let $a'=a/d$.
Note that $d\leq C!$, so $C!a'\geq a$.
We prove that $d\Gamma_{a,b}\leq\Gamma'$. Let $\gamma\in\Gamma_{a,b}$
and let $\la\gamma\ra$ denote the subgroup generated by $\gamma$.
Let $I=[\la\gamma\ra:\la\gamma\ra\cap\Gamma']$.
Since $|\la\gamma\ra|$ divides $a$, $I$ divides $a$. Since $I\leq C$, $I$ divides
$C!$. Hence $I$ divides $d$, which implies that $d\gamma\in\Gamma'$. Since
$d\Gamma_{a,b}\simeq\Gamma_{a',b}$, the lemma follows.
\end{proof}

The next lemma follows easily from Pontryagin duality
(see e.g. \cite[(A3),(A11)]{FrTa}).

\begin{lemma}
\label{lemma:surjection-subgroup}
Let $G,H$ be finite abelian groups. There is a subgroup of $G$ isomorphic to $H$ if and only
if there is a quotient of $G$ isomorphic to $H$.
\end{lemma}

Combining Lemma \ref{lemma:subgroup-isomorphic-to-Gamma-a-b} with the previous result
we immediately obtain the following.

\begin{lemma}
\label{lemma:subgroup-isomorphic-to-Gamma-a-b-quotient}
Let $a,b,C$ be natural numbers and suppose given a surjection
$q:\Gamma_{a,b}\to\Gamma'$ satisfying $|\Ker q|\leq C$. There exists
a subgroup $\Gamma''\leq\Gamma'$ which is isomorphic to $\Gamma_{a',b}$ for
some natural number $a'$ dividing $a$ and satisfying $C!a'\geq a$.
\end{lemma}

\begin{lemma}
\label{lemma:extension-subgroup-Gamma-a-b}
For any natural numbers $b$ and $C_1$ there exists a natural number $C_2$ with the following property.
Suppose that $\Gamma$ is a finite group and that
$N$ is a normal subgroup of $\Gamma$ satisfying $|N|\leq C_1$.
Suppose that $\Gamma/N\simeq\Gamma_{a,b}$ for some natural number $a$. Then there is a subgroup
$\Gamma'\leq\Gamma$ which is isomorphic to $\Gamma_{a',b}$ for some natural number $a'$
satisfying $C_2a'\geq a$.
\end{lemma}
\begin{proof}
By assumption there is a surjective morphism $\pi:\Gamma\to Q:=\Gamma_{a,b}$ whose kernel is $N$.
Let $c:\Gamma\to\Aut N$ be the morphism given by the conjugation action of $\Gamma$ on $N$.
Let $\Gamma_0=\Ker c$. Then $[\Gamma_{a,b}:\pi(\Gamma_0)]\leq [\Gamma:\Gamma_0]\leq |\Aut N|\leq C_1!$.
By Lemma \ref{lemma:subgroup-isomorphic-to-Gamma-a-b} there is a subgroup $Q_0\leq \pi(\Gamma_0)$
which is isomorphic to $\Gamma_{a_0,b}$ for some natural number $a_0$ satisfying $(C_1!)!a_0\geq a$.
The kernel of the restriction $\pi_0$ of $\pi$ to $\Gamma_0$ coincides with the center $Z$ of $N$.
Let $\Gamma_1:=\pi_0^{-1}(Q_0)$. We have an exact sequence
$$0\to Z\to \Gamma_1\stackrel{\pi_0}{\longrightarrow}Q_0\simeq\Gamma_{a_0,b}\to 0.$$
Since this is a central extension, one can define a bilinear morphism
$\beta:Q_0\times Q_0\to Z$
by setting, for any two elements $u,v\in Q_0$,
$\beta(u,v)=[\wt{u},\wt{v}],$
where $\wt{u},\wt{v}\in\Gamma_1$ are arbitrary lifts of $u,v$.
Define a morphism of groups
$$\phi_{\beta}:Q_0\to\Hom(Q_0,Z)$$
setting $(\phi_{\beta}(u))(v)=\beta(u,v)$ for every $u,v\in Q_0$. Now, $Q_0$ can be generated by $b$ elements
because $Q_0\simeq\Gamma_{a_0,b}$, so we may bound
$$|\Hom(Q_0,Z)|\leq |Z|^b\leq |N|^b\leq C_1^b.$$
Consequently,
$Q_1:=\Ker\phi_{\beta}$
satisfies $[Q_0:Q_1]\leq C_1^b$.
By construction, $\pi_0^{-1}(Q_1)$ is an abelian subgroup of $\Gamma_1$.
Using again Lemma \ref{lemma:subgroup-isomorphic-to-Gamma-a-b} we deduce the existence of a subgroup
$Q_2\leq Q_1$ which is isomophic to $\Gamma_{a',b}$ for some natural number $a'$ satisfying
$(C_1^b)!a'\geq a_0$. Then $\pi_0^{-1}(Q_2)$ is a finite abelian group surjecting onto
$\Gamma_{a',b}$, and hence by Lemma \ref{lemma:surjection-subgroup} there is a subgroup
$\Gamma'\leq \pi_0^{-1}(Q_2)\leq\Gamma$ which is isomorphic to $\Gamma_{a',b}$. Setting
$C_2=(C_1^b)!(C_1!)!$ we have
$$C_2a'=(C_1^b)!(C_1!)!a'\geq (C_1!)!a_0\geq a.$$
This finishes the proof of the lemma.
\end{proof}

The next two lemmas refer to finite subgroups of tori. Recall that we use additive notation
for the group structure on tori.

\begin{lemma}
\label{lemma:upper-bound-subgroup-T-d}
Let $\Gamma$ be a finite subgroup of $T^d$ satisfying
$a\gamma=0$ for some natural number $a$ and
every $\gamma\in\Gamma$. Then $|\Gamma|\leq a^d$.
In particular, if $T^d$ contains a subgroup isomorphic to $\Gamma_{a,b}$ for some
$a\geq 2$ then $b\leq d$.
\end{lemma}
\begin{proof}
Let $\pi:\RR^d\to T^d=\RR^d/\ZZ^d$ denote the projection. Then $\pi^{-1}(\Gamma)$ is a discrete
subgroup of $\RR^d$, so it can be generated by $d$ or fewer elements, say $g_1,\dots,g_{d'}$ with
$d'\leq d$, see e.g. \cite[Chap I, Lemma (3.8)]{BtD}. Let $\gamma_i=\pi(g_i)$. Then $\gamma_1,\dots,\gamma_{d'}$ is a generating set of
$\Gamma$, so the morphism $\Gamma_{a,d'}\to\Gamma$ sending $(\ov{\alpha}_1,\dots,\ov{\alpha}_{d'})\in\Gamma_{a,d'}$
to $\sum \alpha_i\gamma_i$ is surjective (here $\alpha_i\in\ZZ$ and $\ov{\alpha}_i$ is the class of $\alpha_i$
in $\ZZ/a$). It follows that $|\Gamma|\leq |\Gamma_{a,d'}|=a^{d'}\leq a^d$.
\end{proof}

\begin{lemma}
\label{lemma:upper-bound-intersection-with-S}
Let $\Gamma$ be a finite subgroup of $T^d$ satisfying
$a\gamma=0$ for some natural number $a$ and
every $\gamma\in\Gamma$.
Let $1\leq j\leq d$ be any integer
and let
$$S=\{(\theta_1,\dots,\theta_d)\in T^d\mid \theta_i=0\text{ for $i\neq j$}\}.$$
Then $|\Gamma\cap S|\geq |\Gamma|/a^{d-1}$.
\end{lemma}
\begin{proof}
We can identify $\Gamma/(\Gamma\cap S)$ with a finite subgroup of $T^d/S\simeq T^{d-1}$,
all of whose elements have order dividing $a$. Hence, by Lemma \ref{lemma:upper-bound-subgroup-T-d},
we have $|\Gamma/(\Gamma\cap S)|\leq a^{d-1}$.
The exact sequence $0\to\Gamma\cap S\to\Gamma\to \Gamma/(\Gamma\cap S)\to 0$
then implies
$$|\Gamma\cap S|=\frac{|\Gamma|}{|\Gamma/(\Gamma\cap S)|}\geq |\Gamma|/a^{d-1},$$
as we wished to prove.
\end{proof}

Recall that the rank of a finite group $G$ is the minimal
size of a generating subset of $G$. We denote the rank of $G$ by $\rk G$.
If $q:G\to H$ is a surjection then clearly $\rk G\geq\rk H$. By Lemma
\ref{lemma:surjection-subgroup}, it follows that if $G$ is a finite abelian
group and $H\leq G$ then $\rk H\leq\rk G$.

The following result was mentioned in the introduction.

\begin{lemma}
\label{lemma:discsym-rank-groups}
Let $X$ be a closed manifold. For any integer $k$ the following are equivalent:
\begin{enumerate}
\item $\discsym(X)\leq k$,
\item exists a constant $C$ such that
every finite abelian group $A$ acting effectively on $X$ has a subgroup $A'\leq A$ satisfying
$[A:A']\leq C$ and $\rk A'\leq k$.
\end{enumerate}
\end{lemma}
\begin{proof}
We prove (1) $\Rightarrow$ (2).
If $\discsym(X)\leq k$ then there exists a natural number $r_0$ such that if
$\Gamma_{r,k+1}$ acts effectively on $X$ then $r\leq r_0$. By \cite[Theorem 2.5]{MS}
there exists a natural number $m_0$ such that, for any prime $p$, if
$\Gamma_{p,m}$ acts effectively on $X$ then $m\leq m_0$.
Let $A$ be a finite abelian group acting effectively on $X$.
There is an isomorphism $\phi:A\to \ZZ/d_1\oplus\dots\oplus\ZZ/d_m$, where
$d_{i+1}$ divides $d_i$ for each $i$ and $d_m\geq 2$. If $m\leq k$, then $\rk A\leq k$.
Suppose $m>k$. Then $\ZZ/d_1\oplus\dots\oplus\ZZ/d_{k+1}$ contains a subgroup isomorphic
to $\Gamma_{d_{k+1},k+1}$, so $d_{k+1}\leq r_0$. Let $p$ be a prime dividing $d_m$.
Then $\ZZ/d_1\oplus\dots\oplus\ZZ/d_m$ has a subgroup isomorphic to $\Gamma_{p,m}$, so
$m\leq m_0$. Let $A'=\phi^{-1}(\ZZ/d_1\oplus\dots\oplus\ZZ/d_k)$. Then $\rk A'=k$
and $[A:A']=|\ZZ/d_{k+1}\oplus\dots\oplus\ZZ/d_m|\leq d_{k+1}^{m-k}\leq C:=r_0^{m_0}$.
The implication (2) $\Rightarrow$ (1) follows
from Lemma \ref{lemma:subgroup-isomorphic-to-Gamma-a-b},
the fact the rank does not increase when passing from a
finite abelian group to a subgroup, and the fact that $\rk \Gamma_{a,b}=b$ for
any $a\geq 2$ and $b$.
\end{proof}

\section{Proof of Theorem \ref{thm:weak-bound-disc-sym}}
\label{ss:proof-thm:weak-bound-disc-sym}

We first recall a few results that will be used in the proof of Theorem \ref{thm:weak-bound-disc-sym}. The following is \cite[Theorem 1.4.14]{AP}.

\begin{theorem}
\label{thm:Carlsson-Baumgartner}
Let $p$ be a prime and let $X$ be a paracompact topological space on which
$\Gamma_{p,m}$ acts freely and trivially on $H^*(X;\ZZ/p)$.
Suppose there exists some $i_0\in\NN$ such that that $H^i(X/G;\ZZ/p)=0$
for all $i\geq i_0$. Then $m\leq |\{j\mid H^j(X;\ZZ/p)\neq 0\}|$.
\end{theorem}

Theorem \ref{thm:Carlsson-Baumgartner}
was originally proved by Gunnar Carlsson in \cite{Carlsson} for $p=2$, and later
by Christoph Baumgartner for odd $p$ in his PhD Thesis.
The following is \cite[Corollary 3.3]{Mundet2022}.

\begin{lemma}
\label{lemma:linearising-continuous}
Let $X$ be a connected $n$-manifold, and suppose that a finite $p$-group $G$
acts continuously and effectively on $X$. If $X^G\neq\emptyset$ then $G$
is isomorphic to a subgroup of $\GL(n,\RR)$.
\end{lemma}

The following theorem is a consequence of a lemma of Minkowski
which states that the size of any finite subgroup of $\GL(n,\ZZ)$ is bounded by
a constant depending only on $n$ (see \cite{Min,Serre}). For details, see
\cite[Lemma 2.6]{Mundet2016}.

\begin{theorem}
\label{thm:Minkowski}
Let $X$ be a compact manifold. There exists a constant $C$ such that, for
every action on $X$ of a finite group $G$, there is a subgroup $G'\leq G$
satisfying $[G:G']\leq C$ whose action on $H^*(X;\ZZ)$ is trivial.
\end{theorem}

Finally, this is \cite[Theorem 1.3]{CMPS}.

\begin{theorem}
\label{thm:CMPS}
Let $X$ be a manifold with finitely generated $H_*(X;\ZZ)$. There exists a constant $C$
such that for every action of a finite $p$-group $G$ on $X$ there is a subgroup $H\leq G$
containing the center of $G$ and satisfying $[G:H]\leq C$ and $|\Stab(H,X)|\leq C$.
\end{theorem}

We now begin the proof of Theorem \ref{thm:weak-bound-disc-sym}.
We will combine two results. This is the first one.

\begin{theorem}
\label{thm:small-primes}
Let $X$ be a closed connected manifold. For every prime $p$ there exists a constant $C(X,p)$ such that if $\Gamma_{p^e,m}$ acts effectively but not freely on $X$
and $m>\dim X$ then $e\leq C(X,p)$.
\end{theorem}
\begin{proof}
Suppose that $G:=\Gamma_{p^e,m}$ acts effectively but not freely on $X$.
By \cite[Corollary 1.5]{CMPS} (see also \cite[Remark 1.6]{CMPS}) there
exists some constant $C'$, depending only on $X$,
such that $U=\{x\in X\mid G_x=\{1\}\}$
satisfies $\dim_{\ZZ/p} H^*(U;\ZZ/p)\leq C'$.
The set $U$ is a proper subset of $X$ because by assumption the action of $G$ on $X$ is not free, and (since $X$ is connected) $U$ does not contain any connected component of $X$. $U$ is also open, so it is a manifold of the same dimension as $X$ with no closed connected component;
consequently $H^i(U;\ZZ/p)=0$ for $i\geq\dim U=\dim X$.
The image of the morphism $G\to \Aut H^*(U;\ZZ/p)$ induced by the action of $G$ on $U$
is a $p$-subgroup, so it is contained
in a Sylow $p$-subgroup of $\Aut H^*(U;\ZZ/p)$. We can identify $\Aut H^*(U;\ZZ/p)$ with a
subgroup of $\GL(C',\ZZ/p)$, which admits as a Sylow $p$-subgroup the group of upper triangular
matrices with $1$'s in the diagonal. The exponent\footnote{Recall that the exponent of a finite
group is the lcm of the orders of the elements of the group.} of such group is $p^{C(p,X)}$, where
$C(p,X):=\lceil \log_pC'\rceil$.
This implies that $K:=p^{C(p,X)}G$ acts trivially on $H^*(U;\ZZ/p)$.
If $e>C(p,X)$ then $K$ contains a subgroup isomorphic to $\Gamma_{p,m}$
so, by Theorem \ref{thm:Carlsson-Baumgartner}, we have $m\leq \dim X$.
\end{proof}

For the second result needed to prove Theorem \ref{thm:weak-bound-disc-sym}
we have to introduce some notation.
Given natural numbers $0<k<n$ define
the polynomial $C_{n,k}(x):=\prod_{i=0}^{k-1}(x^n-x^i)$.
The rational function $Q_{n,k}:=C_{n,k}/C_{k,k}$
is a polynomial, because for any root of unity $\zeta$
the multiplicity of $\zeta$ as a root
of $C_{n,k}$ is not smaller than its multiplicity as a root of $C_{k,k}$,
as the reader can easily check.
Furthermore, $Q_{n,k}$ is monic of degree $k(m-k)$.
%Let $G_{p,m}=(\ZZ/p)^m$.
Let
$\sS_{p,m,k}=\{H\leq \Gamma_{p,m}\mid H\simeq \Gamma_{p,k}\}$. Given $\Gamma\leq \Gamma_{p,m}$
let $\sS_{p,m,k}(\Gamma)=\{H\in \sS_{p,m,k}\mid H\cap\Gamma=0\}$.

\newcommand{\spanvect}{\sigma}

\begin{lemma}
\label{lemma:counting-subgroups}
We have $|\sS_{p,m,k}|=Q_{m,k}(p)$.
If $\Gamma\leq \Gamma_{p,m}$ satisfies $\Gamma\simeq \Gamma_{p,s}$ then
$|\sS_{p,m,k}(\Gamma)|=Q_{m-s,k}(p)p^{ks}$.
\end{lemma}
\begin{proof}
We consider $\Gamma_{p,m}$ as an $m$-dimensional vector space over $\ZZ/p$.
Let $\fF_{p,m,k}=\{(v_1,\dots,v_k)\text{ linearly independent elements of }\Gamma_{p,m}\}$.
The map $\spanvect:\fF_{p,m,k}\to\sS_{p,m,k}$ sending $(v_1,\dots,v_k)$ to its span is surjective, and for every $H\in \sS_{p,m,k}$ we can identify $\spanvect^{-1}(H)$
with $\fF_{p,k,k}$ using any isomorphism $H\simeq \Gamma_{p,k}$. Hence
$|\sS_{p,m,k}|=|\fF_{p,m,k}|/|\fF_{p,k,k}|$, so it suffices to prove
that $|\fF_{p,m,k}|=C_{m,k}(p)$. This follows from induction on $k$.
The initial case $k=1$ is obvious. For the induction step, observe that,
given $(v_1,\dots,v_{k-1})\in \fF_{p,m,k-1}$, the set of $v_k\in \Gamma_{p,m}$
such that $(v_1,\dots,v_{k})\in \fF_{p,m,k}$ is equal to the set
$\Gamma_{p,m}\setminus\spanvect(v_1,\dots,v_{k-1})$, which has $p^m-p^{k-1}$ elements.
This proves the first formula in the lemma.

Now suppose that $\Gamma\leq \Gamma_{p,m}$ satisfies $\Gamma\simeq\Gamma_{p,s}$.
Choose $\Gamma'\leq \Gamma_{p,m}$ such that $\Gamma_{p,m}=\Gamma\oplus\Gamma'$
and pick an isomorphism $f:\Gamma'\to \Gamma_{p,m-s}$. Let $\pi:\Gamma_{p,m}=\Gamma\oplus\Gamma'\to\Gamma'$ be the projection.
For each $H\in \sS_{p,m,k}(\Gamma)$, $\pi(H)\leq\Gamma'$ is isomorphic
to $\Gamma_{p,k}$, so $f\circ\pi$ defines a map $\phi:\sS_{p,m,k}(\Gamma)\to\sS_{p,m-s,k}$.
The map $\phi$ is surjective, and given $K\in \sS_{p,m-s,k}$ we can identify
$\phi^{-1}(K)$ with the set $\mM$ of linear maps $h:f^{-1}(K)\to\Gamma$ (by associating
to $h$ its graph). Since $|\mM|=p^{ks}$, we obtain
the desired formula for $|\sS_{p,m,k}(\Gamma)|$.
\end{proof}

For each $0<k<m$ and $s\leq m-k$ the polynomial $R_{m,k,s}:=Q_{m,k}-Q_{m-s,k}x^{ks}$
has degree less than $k(m-k)$. For each
$\Gamma\leq \Gamma_{p,m}$ let $\dD_{p,m,k}(\Gamma)=\sS_{p,m,k}\setminus\sS_{p,m,k}(\Gamma)$.
By the previous lemma we have $|\dD_{p,m,k}(\Gamma)|=R_{m,k,s}(p)$, where
$\Gamma\simeq \Gamma_{p,s}$.

\begin{theorem}
\label{thm:big-primes}
Let $X$ be a closed connected $n$-dimensional manifold and let $m=[3n/2]+1$.
There exists a constant
$C'(X)$ such that for every prime $p\geq C'(X)$ and any effective action of
$\Gamma_{p,m}$ on $X$ there exists a subgroup $H\leq \Gamma_{p,m}$ isomorphic to $\Gamma_{p,n+1}$
which intersects trivially all stabilizers of the action of $\Gamma_{p,m}$ on $X$.
\end{theorem}
\begin{proof}
Let $C$ be the constant given by applying Theorem \ref{thm:CMPS} to $X$.
Let $m=[3n/2]+1$. For each $1\leq s\leq [n/2]$ there exists some $C_s$ such that
for every $t\geq C_s$ we have $R_{m,n+1,s}(t)<C^{-1}Q_{m,n+1}(t)$, because
$Q_{m,n+1}$ is monic and $\deg R_{m,n+1,s}<\deg Q_{m,n+1}$.
We claim that $C'(X):=3+\max\{C_s\mid 1\leq s\leq [n/2]\}$ has the desired property.
Indeed, suppose that $p\geq C'(X)$ is a prime and $G:=\Gamma_{p,m}$ acts effectively on $X$.
By Theorem \ref{thm:CMPS} we have $|\Stab(G,X)|\leq C$. For every $K\in\Stab(G,X)$
there exists some $x\in X$ such that $G_x=K$, so, by Lemma  \ref{lemma:linearising-continuous},
$K$ is isomorphic to a subgroup of $\GL(n,\RR)$.
Since $p\geq 3$, this implies that $\dim K\leq [n/2]$. Hence, by Lemma \ref{lemma:counting-subgroups} we have
$|\dD_{p,m,n+1}(\Gamma)|=R_{m,n+1,s}(p)<C^{-1}Q_{m,n+1}(p)$, which implies that
$\bigcup_{K\in\Stab(G,X)}\dD_{p,m,n+1}(K)\neq\sS_{p,m,n+1}$. Consequently there
exists some $H\in \sS_{p,m,n+1}$ that does not belong to $\dD_{p,m,n+1}(K)$ for
any $K\in\Stab(G,X)$. Equivalently, $H$ intersects trivially each $K\in\Stab(G,X)$.
\end{proof}

We are now ready to prove Theorem \ref{thm:weak-bound-disc-sym}.
Let $X$ be a closed and connected $n$-manifold. Let $m=[3n/2]+1$.
Arguing by contradiction,
suppose that there exists a sequence of integers $r_i\to\infty$ and an
effective action of $\Gamma_{r_i,m}$ on $X$ for each $i$.

Let $\pP=\{p\text{ prime}\mid p\text{ divides $r_i$ for some $i$}\}$.
We distinguish two possibilities. If $\pP$ is infinite, then we can take a sequence
of primes $p_j$ belonging to $\pP$ and satisfying $p_j\to\infty$. Each
$p_j$ divides $r_{i_j}$ for some $i_j$, so $\Gamma_{p_j,m}$ is isomorphic to
a subgroup of $\Gamma_{r_{i_j},m}$ and hence by restricting the action of the
latter to the former we obtain, for each $j$, an effective action of
$\Gamma_{p_j,m}$ on $X$ for each $j$. By Theorem \ref{thm:big-primes}
we get a free action of $\Gamma_{p_j,n+1}$ on $X$ for big enough $j$.
By Theorem \ref{thm:Minkowski}, if $j$ is big enough then
the action of $\Gamma_{p_j,n+1}$ on $H^*(X;\ZZ)$,
and hence on $H^*(X;\ZZ/p_j)$, is trivial. This contradicts
Theorem \ref{thm:Carlsson-Baumgartner}.

The second possibility is that $\pP$ is bounded. In that case, there exists some
$p\in\pP$ and a sequence of natural numbers $e_j\to\infty$ such that $p^{e_j}$
divides $r_{i_j}$ for some $i_j$. Arguing as before, this gives an effective
action of $\Gamma_{p^{e_j},m}$ on $X$ for each $j$. This contradicts
Theorem \ref{thm:small-primes}, so the proof of Theorem \ref{thm:weak-bound-disc-sym}
is finished.

\section{Equivariant maps to the torus}
%: proof of Theorem \ref{thm:linearising-action}}
\label{s:equivariant-maps-torus}

In all this section $X$ denotes a closed, connected and oriented
$n$-dimensional manifold.
We identify $T^n$ with the quotient $\RR^n/\ZZ^n$,
and we use additive notation for the group structure on $T^n$.
Suppose that $G$ is a group,
$\eta:G\to T^n$ is a group homomorphism,
and $G$ acts on a space $X$. A map $\phi:X\to T^n$ will be called {\it $\eta$-equivariant}
if it satisfies $\phi(g\cdot x)=\eta(g)+\phi(x)$ for every $x\in X$ and $g\in G$.

\begin{theorem}
\label{thm:linearising-action}
Let $X$ be a closed, connected and oriented $n$-dimensional topological manifold.
Let $\phi:X\to T^n$ be a continuous map of nonzero degree. Let $G$ be a finite group.
Suppose that $X$ is endowed with an effective action
of $G$ inducing the trivial action on $H^1(X;\ZZ)$. Then there is a morphism of groups $\eta:G\to T^n$ with these properties:
\begin{enumerate}
\item the map $\phi$ is homotopic to an $\eta$-equivariant map $\psi:X\to T^n$,
\item %\marginpar{abans hi havia desigualtat, ara divisibilitat; fer canvis on calgui}
    $|\Ker\eta|$ divides $\deg\phi$.
\end{enumerate}
\end{theorem}

Before proving the theorem we prove three auxiliary  lemmas.
The first lemma is a topological
analogue of the construction at the beginning of \cite[\S 2.1]{M1}.
We identify $S^1$ with $\RR/\ZZ$ and accordingly we use additive notation
for the group structure on $S^1$.

\begin{lemma}
\label{lemma:equivariant-map-to-circle}
Let $\alpha:X\to S^1$ be a continuous map. Let $\theta$ be a generator of $H^1(S^1;\ZZ)$.
Suppose that a finite group $G$ acts continuously on
$X$ preserving $\alpha^*\theta$. Let $r$ be the cardinal of $G$ and let $\mu_r\subset S^1$ denote
the group of $r$-th roots of unity.
There exists a morphism of groups $\xi:G\to\mu_r$
and a continuous map $\beta:X\to S^1$ homotopic to $\alpha$ such that $\beta(g\cdot x)=\xi(g)+\beta(x)$
for every $x\in X$ and $g\in G$.
\end{lemma}
\begin{proof}
Define $\zeta:X\to S^1$ by $\zeta(x)=\sum_{g\in G}\alpha(g\cdot x)$ for every $x\in X$.
Then $\zeta$ is continuous and constant on $G$-orbits.
Let $\rho_g:X\to X$ be the homeomorphism induced by the action of $g\in G$. By assumption $\rho_g^*\alpha^*\theta=\alpha^*\theta$ for every $g\in G$. We have $\zeta^*\theta=\sum_{g\in G} \rho_g^*\alpha^*\theta=r\alpha^*\theta$.

Let $\Gamma=\{(x,t)\in X\times S^1\mid \zeta(x)=rt\}$.
Let $\pi:\Gamma\to X$ be the restriction of the projection map $X\times S^1\to X$. The action of $\mu_r$ on $\Gamma$ given
by $(x,t)\cdot\theta=(x,t\theta)$ endows $\pi:\Gamma\to X$ with a structure of principal $\mu_r$-bundle.
We claim that it is a trivial principal bundle. This is equivalent to the triviality of the monodromy of $\pi$,
which we denote by $\nu:\pi_1(X,x_0)\to\mu_r$, where $x_0\in X$ is an arbitrary base point.
If there existed some $\lambda\in\pi_1(X,x_0)$ such that $\nu(\lambda)\neq 0$ then the pairing
of $\zeta^*\theta$ with $[\lambda]\in H_1(X;\ZZ)$ would not be divisible by $r$, which contradicts
the fact that $\zeta^*\theta=r\alpha^*\theta$. Hence $\nu$ is trivial and consequently the bundle $\pi:\Gamma\to X$
is trivial, so we may choose a section $\sigma:X\to\Gamma$.

Define $\beta:X\to S^1$ by the condition
that $\sigma(x)=(x,\beta(x))$. Then $\beta$ is continuous and we have $r\beta(x)=\zeta(x)$ for every $x\in X$. For any $g\in G$ define
$\chi_g:X\to S^1$ by $\chi_g(x)=\beta(g\cdot x)-\beta(x)$. We have
$r\chi_g(x)= r\beta(g\cdot x)-r\beta(x)=\zeta(g\cdot x)-\zeta(x)=0$ because $\zeta$ is $G$-invariant.
Hence $\chi_g$ takes values in $\mu_r$, and consequently, being continuous, it is a constant map. We may thus define a map $\xi:G\to\mu_r$ by the condition that
$\xi(g)=\chi_g(x)$ for every $x\in X$. Let $g,g'\in G$ and let $x\in X$.
We have
$$\chi_{gg'}(x)=\beta(gg'\cdot x)-\beta(x)
=\beta(gg'\cdot x)-\beta(g'\cdot x)+\beta(g'\cdot x)-\beta(x)
=\chi_g(g'\cdot x)+\chi_{g'}(x),$$
which proves that $\xi(gg')=\xi(g)+\xi(g')$, so $\xi$ is a morphism of groups. Now the formula
$\beta(g\cdot x)=\xi(g)+\beta(x)$ follows immediately from the definition
of $\xi$. To conclude the proof, note that
$r\beta^*\theta=\zeta^*\theta=r\alpha^*\theta$, so $r(\beta^*\theta-\alpha^*\theta)=0$. Since $H^1(X;\ZZ)$ has no torsion
we conclude that $\beta^*\theta=\alpha^*\theta$. Hence
$\beta$ and $\alpha$ are homotopic, because $S^1$ is a model for
$K(\ZZ,1)$.
\end{proof}

\begin{lemma}
\label{lemma:X-star-connected}
Let $G$ be a finite group acting effectively, continuously and preserving the orientation on $X$. Let $X^*\subseteq X$ be the set of points with trivial stabilizer. Then $X^*$ is connected.
\end{lemma}
\begin{proof}
For any $g\in G$ denote $X^g=\{x\in X\mid g\cdot x=x\}$.
Let $g_1,\dots,g_s$ be the (nontrivial) elements of $G$ of prime order.
Since any nontrivial element of $G$ has some power belonging to
the set $\{g_1,\dots,g_s\}$, we have $X^*=X\setminus\bigcup_i X^{g_i}$.
Define $X_1=X$ and
$X_i=X\setminus (X^{g_1}\cup\dots\cup X^{g_{i-1}})$ for $2\leq i\leq s+1$.
Then $X_i$ is an open subset of $X$ for each $i$.
We prove that $X_i$ is connected for every $1\leq i\leq s+1$, using ascending induction on $i$. Since $X^*=X_{s+1}$, this will imply the lemma.
Clearly $X_1$ is connected. Now suppose that $1\leq i\leq s$ and that $X_i$ is connected. Let $p_i$ be the order of $g_i$.
Since $G$ acts on $X$ preserving the orientation, by
\cite[Chap V, Theorem 2.3]{Bo}
and
\cite[Chap V, Theorem 2.5]{Bo},
$X^{g_i}$ is a $\ZZ/p_i$-cohomology manifold
of dimension $d_i$, where $n-d_i$ is even.
Arguing as in the proof of \cite[Chap V, Theorem 2.6]{Bo}
we conclude that $d_i<n$, and consequently $d_i\leq n-2$.
Applying
\cite[Chap I, Corollary 4.7]{Bo} we conclude that
$X_{i+1}=X_i\setminus (X_i\cap X^{g_i})$ is connected, so the lemma is proved.
\end{proof}

The following result generalizes \cite[Lemma 2.5]{DS}.

\begin{lemma}
\label{lemma:degree-quotient}
Suppose that a finite group $G$ acts effectively, continuously, and
preserving the orientation on $X$. Let $\pi:X\to X/G$ denote the quotient
map. Let $r$ denote the cardinal of $G$. The image of the map $\pi^*:H^n(X/G;\ZZ)\to H^n(X;\ZZ)$ is contained in $r H^n(X;\ZZ)$.
\end{lemma}
\begin{proof}
Denote as in the previous lemma by $X^*$ the open subset of $X$ consisting of points with trivial stabilizer. Since $X/G$ is endowed with the quotient topology, $\pi(X^*)=X^*/G$ is an open subset of $X/G$. Let
$F=X\setminus X^*$.
Consider the following commutative diagram, where $H_c^*(\cdot;\ZZ)$ denotes cohomology
with compact support,
the rows are portions of the long exact sequences
for the inclusions
$X^*\hookrightarrow X\hookleftarrow F$ and
$X^*/G\hookrightarrow X/G \hookleftarrow F/G$ (see e.g. \cite[Chap I, (2) in \S1.1]{Bo}),
and the vertical arrows are pullback morphisms induced by proper maps:
$$\xymatrix{  H^n_c(X^*;\ZZ) \ar[r]^j %\ar[r]_{j_{X,X^*}}
& H^n_c(X;\ZZ) \ar[r]^{r} & H^n_c(F;\ZZ) \ar[r] & 0\\
H^n_c(X^*/G;\ZZ)\ar[u]^{\pi^*} \ar[r]^-{j_G} %\ar[r]_{j_{X/G,X^*/G}}
& H^n_c(X/G;\ZZ) \ar[u]^{\pi^*} \ar[r]^{r_G} %\ar[r]_{r_{X/G,X^*/G}}
& H^n_c(F/G;\ZZ)\ar[u]\ar[r] & 0,}$$
The morphism $j$ is an isomorphism because, by Lemma \ref{lemma:X-star-connected},
$X^*$ is connected (see \cite[Chap I, Theorem 4.3]{Bo}). By the exactness this implies that
$H^n_c(F;\ZZ)=0$. The long exact sequence for $X^*\hookrightarrow X\hookleftarrow F$ and the
fact that $\dim X=n$ imply that $H^k_c(F;\ZZ)=0$ for every $k\geq n$
(here we are using \cite[Chap I, (3) in \S 1.2]{Bo}).
From \cite[Chap III, Theorem 5.2]{Bo} it follows that
$H^k_c(F/G;\ZZ)$ for every $k\geq n$. Hence $j_G$ is surjective. Consequently it suffices to prove that the image of the morphism
$\pi^*:H^n_c(X^*/G;\ZZ)\to H^n_c(X^*;\ZZ)$ is contained in $r H^n_c(X^*;\ZZ)$.

Denote $G^*=G\setminus\{1\}$. We are going to prove that there exists a connected open subset $U\subset X^*$ such
that $gU\cap U=\emptyset$ for every $g\in G^*$. Fix some point $x\in X^*$. Since
$X^*$ is Hausdorff, for every $g\in G^*$ there exists disjoint open subsets $A_g,B_g\subset X^*$ such
that $x\in A_g$ and $gx\in B_g$. Let $C=\bigcap_{g\in G^*}A_g$. Then $gx\notin \ov{C}$ for every $g\in G^*$,
because $C\cap B_g=\emptyset$. This implies that $x$ belongs to the open set $D:=C\setminus\bigcup_{g\in G^*}g\ov{C}$.
Let $U\subset D$ be a connected open subset containing $x$. Then for every $g\in G^*$ we have
$gU\subset gC$ because $U\subset C$, and consequently $gU\cap D=\emptyset$, which implies that
$gU\cap U=\emptyset$.

Let $V=\pi(U)$, so that $\pi^{-1}(V)=GU=\bigcup_{g\in G}gU$.
Consider the following commutative diagram:
$$\xymatrix{H_c^n(GU;\ZZ)\ar[r]^{j_{GU}} & H_c^n(X^*;\ZZ) \\
H_c^n(V;\ZZ)\ar[u]^{\pi_V^*}\ar[r]^-{j_V} & H_c^n(X^*/G;\ZZ) \ar[u]^{\pi^*},}$$
where $j_{GU}$ and $j_V$ are the covariant morphisms induced by open embeddings and the vertical
arrows are pullback morphisms induced by proper morphisms.
By \cite[Chap I, Theorem 4.3]{Bo} $j_V$ is an isomorphism, so it suffices to prove that
the image of $j_{GU}\circ\pi_V^*$ is contained in $r H_c^n(X^*;\ZZ)$.
Since $gU\cap U=\emptyset$ for every
$g\in G^*$, the open subset $GU$ contains $r=|G|$ connected components,
which are $\{gU\mid g\in G\}$. Denote by $i_g:gU\to GU$
the inclusion. The pullback morphisms $i_g^*:H_c^n(GU;\ZZ)\to H_c^n(gU;\ZZ)$ combine to give an isomorphism
$H_c^n(GU;\ZZ)\stackrel{\simeq}{\longrightarrow}\bigoplus_{g\in G}H_c^n(gU;\ZZ)$. Let $j_{gU}:H_c^n(gU;\ZZ)\to H_c^n(X^*;\ZZ)$
be the morphism induced by the open embedding $gU\hookrightarrow X^*$. We have
$j_{GU}=\sum_{g\in G}j_{gU}\circ i_g^*$, so if we prove that
$j_{gU}\circ i_g^*\circ \pi_V^*=j_{hU}\circ i_h^*\circ \pi_V^*$ for every $g,h\in G$ then
we will be done. Take two elements $g,h\in G$ and let $\rho:X^*\to X^*$ be the map given
by $\rho(x)=gh^{-1}\cdot x$. Then $\rho$ is a homeomorphism and it restricts to a homeomorphism
$\rho:hU\to gU$. The induced morphism $\rho^*:H_c^n(X^*;\ZZ)\to H_c^n(X^*;\ZZ)$
is the identity because $G$ acts on $X$ (and hence on $X^*$) preserving the orientation.
Now, the desired equality follows from the commutativity of the following diagram:
$$\xymatrix{ & H_c^n(hU;\ZZ)\ar[r]^{j_{hV}} & H_c^n(X^*;\ZZ) \\
H_c^n(V;\ZZ)\ar[ur]^{i_j^*\circ\pi_V^*} \ar[dr]_{i_g^*\circ \pi_V^*} \\
& H_c^n(gU;\ZZ)\ar[uu]_{\rho^*} \ar[r]^{j_{gU}} & H_c^n(X^*;\ZZ)\ar[uu]_{\rho^*=\Id}}$$
The triangle commutes because $\pi_V\circ i_h=\pi_V\circ i_g\circ \rho$ and the square
commutes because $\rho$ is a homeomorphism and the inclusion $hU\hookrightarrow X^*$
is equal to the composition of the inclusion $gU\hookrightarrow X^*$
and $\rho$.
\end{proof}

We are now ready to prove Theorem \ref{thm:linearising-action}. Let
$\phi_i:X\to S^1$ be the composition of $\phi$ with the projection
to the $i$-th factor $T^n=(S^1)^n\to S^1$. Since $G$ acts trivially
on $H^1(X;\ZZ)$, in particular it fixes $\phi_i^*\theta$, where
$\theta\in H^1(S^1;\ZZ)$ is any generator. Applying Lemma \ref{lemma:equivariant-map-to-circle}
to $\phi_i$ we obtain the existence of a morphism of groups $\eta_i:G\to S^1$ and
a map $\psi_i:X\to S^1$ homotopic to $\phi_i$ such that $\psi_i(g\cdot x)=\eta_i(g)+\psi_i(x)$
for every $x\in X$ and $g\in G$. Define
$\psi=(\psi_1,\dots,\psi_n):X\to T^n$ and $\eta=(\eta_1,\dots,\eta_n):G\to T^n$.
Then $\psi$ is homotopic to $\phi$ and it is $\eta$-equivariant.

Let $G_0=\Ker\eta$. The map $\psi$ factors as a composition
$$X\stackrel{\pi}{\longrightarrow} X/G_0\stackrel{\psi_0}{\longrightarrow} T^n,$$
where $\pi$ is the natural quotient map. Hence $\psi^*=\pi^*\circ\psi_0$,
so the image of $\psi^*:H^n(T^n;\ZZ)\to H^n(X;\ZZ)$ is contained in the
image of $\pi^*:H^n(X/G_0;\ZZ)\to H^n(X;\ZZ)$. By the definition of degree,
the image of $\psi^*:H^n(T^n;\ZZ)\to H^n(X;\ZZ)$ is equal to $(\deg \psi)H^n(X;\ZZ)$,
and by Lemma \ref{lemma:degree-quotient} the image of $\pi^*$ is contained in
$|G_0|\cdot H^n(X;\ZZ)$. It then follows that $|G_0|$ divides $\deg\psi$.
But $\deg\psi=\deg\phi$ because $\psi$ and $\phi$ are homotopic. So the proof
of the theorem is now complete.

\begin{remark}
The previous results have been independently proved by Csik\'os, Pyber and Szab\'o,
lifting $\phi$ to a map $\zeta$ from the universal cover of $X$ to that of $T^n$, and defining
$\psi$ as the average the translates of
$\zeta$ by lifts of the action of elements of $G$ to the universal covers of $X$ and $T^n$.
\end{remark}

\section{Some consequences of Theorem \ref{thm:linearising-action}}
\label{s:consequences-thm:linearising-action}
%We next explain how Theorem \ref{thm:linearising-action} implies
%Theorems \ref{thm:CMPS-hypertoral} and Theorem \ref{thm:Jordan}.

Let $X$ be a closed, connected and
oriented $n$-dimensional manifold.
Choose a continuous map $\phi:X\to T^n$ satisfying
$|\deg\phi|=\min\{|\deg\psi|\mid \psi:X\to T^n,\,\deg\psi\neq 0\}$.

Suppose that a finite group $G$ acts continuously
on $X$. Let $G'$ be the kernel of the induced morphism $G\to\Aut H^1(X;\ZZ)$.
By Theorem \ref{thm:Minkowski} we have $[G:G']\leq C_X$ for some
constant $C_X$ depending only on $X$. By Theorem \ref{thm:linearising-action}
there is a morphism of groups $\eta:G'\to T^n$ satisfying
$|\Ker\eta|\leq |\deg\phi|$ and an $\eta$-equivariant map $\psi:X\to T^n$
homotopic to $\phi$.
The existence of the $\eta$-equivariant map $\psi$ implies that
for every $x\in X$ we have $G_x\leq\Ker\eta$
(if $g\in G_x$ then $\eta(g)=\psi(g\cdot x)-\psi(x)=\psi(x)-\psi(x)=0$), so the previous
bound implies that $|\Stab(X,G)|\leq 2^{|\deg\phi|}$. This proves
Theorem \ref{thm:CMPS-hypertoral}.

Suppose that the group in the previous argument is $G=\Gamma_{a,b}$. By Lemma
\ref{lemma:subgroup-isomorphic-to-Gamma-a-b}
there is a subgroup of $G'$ isomorphic to $\Gamma_{a',b}$ where $a'$ divides $a$ and
$a'\geq a/C_X!$. By Lemma \ref{lemma:upper-bound-subgroup-T-d}
we have $|\eta(\Gamma_{a',b})|\leq a^n$, and hence
$(a/C_X!)^b\leq |\Gamma_{a',b}|\leq |\Ker\eta|\cdot a^n\leq |\deg\phi|a^n$. Since $|\deg\phi|$
only depends on $X$ (and not on $a$ and $b$),
we conclude that, assuming $a$ is big enough, $b\leq n$. This implies
statement (1) in Theorem \ref{thm:discsym-rationally-hypertoral}.

To prove Theorem \ref{thm:Jordan} note that $\eta(G')$, being a subgroup of $T^n$,
is abelian and can be generated by $n$ or fewer elements.
If $\deg\phi=1$ then $G'\simeq\eta(G')$, so
$G'$ is abelian. This implies that $\Homeo(X)$ is Jordan in this case.
For other values of $\deg\phi$, we
apply \cite[Lemma 2.2]{M1} to the exact sequence
$$1\to\Ker\eta\to G'\to \eta(G')\to 1$$
and since $|\Ker\eta|$ divides $\deg\phi$
we deduce the existence of an abelian subgroup $G''\leq G'$ such that
$[G':G'']$ is bounded above by a constant depending only on $n$ and $\deg \phi$.
Since $n$ and $\deg\phi$ only depend on $X$,
$[G:G'']$ is bounded above by a constant depending only on $X$.
Hence $\Homeo(X)$ is Jordan.

If $\deg\phi=1$ Theorem \ref{thm:Jordan} can also be proved using \cite[Theorem 2.5]{GLO}.

Theorem \ref{thm:linearising-action} can also be used to give examples
of rationally hypertoral manifolds which are not hypertoral. Let $k\geq 2$ and $n\geq 2$
be integers. Let $\pi_L:L\to T^n$ be a complex line bundle, and let $\sigma$ be a smooth section
of $L^{\otimes k}$ intersecting transversely and nontrivially the zero section.
Let $Y=\{v\in L\mid v^{\otimes k}=\sigma(\pi(v))\}$. Then $Y$ is a smooth connected
$n$-manifold. Let $G$ be the group of $k$-th roots of unity. The action of $G$
on $L$ by multiplication preserves $Y$. Let $\gamma\in G$ be a generator.
Let $X=(Y\times\RR)/\sim$ where $(v,t+1)\sim (\gamma\cdot v,t)$.
Let $\pi_R:\RR\to\RR/\ZZ$ be the projection. Then
$\pi_L\times\pi_R:Y\times\RR\to T^n\times\RR/\ZZ=T^{n+1}$ descends
to a continuous map $X\to T^{n+1}$ of degree $k$. Hence $X$ is rationally hypertoral.
The following theorem implies that $X$ is not hypertoral.

\begin{theorem}
\label{thm:non-hypertoral}
For any continuous map $\phi:X\to T^{n+1}$ the degree of $\phi$ is divisible by $k$.
\end{theorem}
\begin{proof}
The diagonal action of $G$ on $Y\times\RR$ (trivial on $\RR$)
descends to an action of $G$ on $X$ with nonempty fixed point set
(because $\sigma^{-1}(0)\neq\emptyset$). Let $\gamma^*:H^*(Y;\QQ)\to H^*(Y;\QQ)$
be induced by the action of $\gamma$. For every $k$ there is a
$G$-equivariant exact sequence
$$0\to \Coker (1-\gamma^*)|_{H^{k-1}(Y;\QQ)}
\to H^k(X;\QQ)
\to \Ker(1-\gamma^*)|_{H^k(Y;\QQ)}\to 0.$$
The actions of $G$ on
$\Coker (1-\gamma^*)|_{H^{k-1}(Y;\QQ)}$ and
$\Ker(1-\gamma^*)|_{H^k(Y;\QQ)}$ are trivial.
Since $G$ is finite, it follows that the action of $G$ on $H^k(X;\QQ)$ is trivial.
Since $H^1(X;\ZZ)$ is torsion free, the action of $G$ on $H^1(X;\ZZ)$ is also trivial.
Let $\phi:X\to T^{n+1}$ be any continuous map. By Theorem \ref{thm:linearising-action}
there is a morphism $\eta:G\to T^{n+1}$ such that $\phi$ is homotopic to an $\eta$-equivariant map $X\to T^{n+1}$. Since the fixed point set of $G$ is nonempty, $\eta$ is necessarily the trivial
map. It follows that $\deg \phi$ is divisible by $|\Ker\eta|=|G|=k$.
\end{proof}

\section{Finitely generated $\ZZ[t_1^{\pm 1},\dots,t_1^{\pm 1}]$-modules}
\label{s:commutative-algebra}

\begin{theorem}
\label{thm:finitely-generated}
Let $A$ be a Noetherian ring and let $M$ be a finitely generated
$A[z]$-module. Suppose that there exists a sequence of integers
$r_j\to\infty$ and
$A[z]$-module morphisms $w_j:M\to M$ such that $w_j^{r_j}$ coincides
with multiplication by $z$. Then $M$ is finitely generated as an $A$-module.
\end{theorem}
\begin{proof}
Let $S\subset M$ be a finite $A[z]$-generating set. Let $M_0\subseteq M$
be the $A$-submodule generated by $S$. Define an increasing sequence of
$A$-submodules of $M$
$$M_0\subseteq M_1\subseteq\dots\subseteq M_d\subseteq\dots $$
by the condition that $M_d=M_{d-1}+zM_{d-1}$ for every positive integer
$d$. Define also
$M_d=0$ for negative integers $d$. For each $d$ the quotient $M_d/M_{d-1}$ is a finitely generated
$A$-module, because $M_d$ is a finitely generated $A$-module.

For any $d\geq 1$, multiplication
by $z$ gives a surjective morphism
$$\mu_d:M_{d-1}/M_{d-2}\to M_{d}/M_{d-1}.$$
Consider the composition
$$\nu_d=\mu_d\circ\dots\circ\mu_1:M_0\to M_d/M_{d-1},$$
and define $K_d=\Ker \nu_d$.
Each $K_d$ is an $A$-submodule
of $M_0$, and there are inclusions
$$K_0\subseteq K_1\subseteq K_2\subseteq\dots.$$
Since $A$ is Noetherian and $M_0$ is a finitely generated $A$-module, there exists
some $d_0$ such that $K_d=K_{d-1}$ for $d\geq d_0$.
If for some $d$ the morphism $\mu_d$ fails to be an isomorphism, then $\Ker\nu_d$
is strictly bigger than $\Ker\nu_{d-1}$, because $\nu_d=\mu_d\circ\nu_{d-1}$ and
$\nu_{d-1}$ is surjective. It follows that $\mu_d$ is an isomorphism for any
$d\geq d_0$.

Let $N=M_{d_0}/M_{d_0-1}$. If $N=0$ then $M=M_{d_0-1}$, so $M$ is finitely generated
as an $A$-module and we are done.

Suppose from now on that $N\neq 0$. Since $A$ is Noetherian and
$N$ is finitely generated, there exists
a filtration by $A$-submodules
\begin{equation}
\label{eq:filtration-N}
0=N_0\subset N_1\subset\dots\subset N_r=N
\end{equation}
in such a way that $N_j/N_{j-1}\simeq A/\plie_j$ for primes $\plie_1,\dots,\plie_r\in\Spec A$
(see \cite[Chap 7, Exercise 18]{AM} or \cite[Theorem 6.4]{Mat}).
Let $\plie$ be a minimal element of $\{\plie_1,\dots,\plie_r\}$. Denote as usual by
$A_{\plie}$ the localisation of $A$ at $\plie$ and by $k_{\plie}$ its residual field.
For any $A$-module $R$ we denote $R_{\plie}=R\otimes_A A_{\plie}$.
Since $A_{\plie}$ is a flat $A$-module, for any inclusion of $A$-modules
$R'\subseteq R$ we have $R_{\plie}/R_{\plie}'\simeq (R/R')_{\plie}$.
Since $\plie$ is a minimal element of $\{\plie_1,\dots,\plie_r\}$, for every $i$ we have
$$(A/\plie_i)_{\plie}\simeq\left\{\begin{array}{cc}
k_{\plie} & \qquad\text{if $\plie_i=\plie$,} \\
0 & \qquad \text{if $\plie_i\neq \plie$.}
\end{array}\right.$$
Hence, $(A/\plie_i)_{\plie}$ is a simple $A_{\plie}$-module for every $i$.
So if
we tensor by $A_{\plie}$ the elements of the filtration (\ref{eq:filtration-N})
and we ignore the resulting inclusions that are actually equalities, we get
a composition series for $N_{\plie}$ of length
$$\lambda:=|\{i\mid \plie_i=\plie\}|\geq 1$$
(see e.g. the paragraph before Proposition 6.7 in \cite{AM}).

To conclude the proof of the theorem we are going to prove that
there is no $A[z]$-module morphism $w:M\to M$ satisfying $w^r=z$
for any $r>\lambda$. Arguing by contradiction, let us assume that
there exists an $A[z]$-module morphism $w:M\to M$ satisfying $w^r=z$ for some
$r>\lambda$.

Let $M_0'=M_0$ and define recursively $M_{\delta}'$ for positive integers
$\delta$ as $M_{\delta}'=M_{\delta-1}+w M_{\delta-1}$. Define also $M_{\delta}'=0$
for negative integers $\delta$. The action of $w$ defines a surjective
$A$-module morphism $\mu'_{\delta}:M'_{\delta-1}/M'_{\delta-2}\to M'_{\delta}/M'_{\delta-1}$.
Arguing as we did for $M_d$, we prove the existence of some $\delta_0$ such that
$\mu'_{\delta}$ is an isomorphism for any $\delta\geq\delta_0$.
Let $N'=M'_{\delta_0}/M'_{\delta_0-1}$.

Denote by $A[z]_{\leq d}$ (resp. $A[w]_{\leq\delta}$) the $A$-module of
polynomials in $z$ (resp. $w$) of degree at most $d$ (resp. $\delta$).
We have
\begin{equation}
\label{eq:polynomials}
M_{d}=A[z]_{\leq d}M_0,\qquad M_{\delta}'=A[w]_{\leq\delta}M_0.
\end{equation}
Since $w^r=z$, we have
$A[z]_{\leq d}\subseteq A[w]_{\leq rd}$ for every $d$. This implies that
$$M_d\subseteq M'_{rd}$$
for every nonnegative $d$.
Suppose that $S=\{m_1,\dots,m_s\}$.
Since $m_1,\dots,m_s$ generate $M$ as an $A[z]$-module, there exist
polynomials $P_{ijk}\in A[z]$ for $i=1,\dots,k-1$ and $j,k=1,\dots,s$
such that
$$w^im_j=P_{ij1}m_1+\dots+P_{ijr}m_s.$$
Let $e=\max_{i,j,k}\deg P_{ijk}$. We have
$w^im_j\in M_e$ for every $i=1,\dots,k-1$ and $j=1,\dots,s$, and this implies that for
any $d$ we have $A[w]_{\leq\delta}M_0\subseteq A[z]_{\leq [\delta/r]+e}M_0$, or equivalently
$$M_{\delta}'\subseteq M_{[\delta/r]+e}.$$

Following \cite[Chap 6]{AM} (see the proof of \cite[Proposition 6.7]{AM}) for any
$A_{\plie}$-module $R$ we denote by $l(R)$ the length of $R$.
This is defined to be $\infty$ if $R$ has no composition series of finite length and
it is equal to $n$ if $R$ has a composition series of lenght $n$. This is well defined
by \cite[Proposition 6.7]{AM}. Furthermore, for any inclusion of $A_{\plie}$-modules
$R\subseteq R'$ we have $l(R')=l(R)+l(R'/R)$ (see \cite[Proposition 6.9]{AM}).
For example, we have $l(N_{\plie})=\lambda$, and if $d\geq d_0$ then
$l((M_{d+k})_{\plie}/(M_d)_{\plie})=k\lambda$ for every $k$.

To simplify our notation we will denote
$M_{i,\plie}=(M_i)_{\plie}$ and $M'_{i,\plie}=(M'_i)_{\plie}$ for every $i$.
Fix some value of $d$ satisfies both $d\geq d_0$ and $[d/r]\geq\delta_0$.
Let $k$ be a big number, to be specified later. The inclusions
$M_{d+k,\plie}\subseteq M'_{r(d+k),\plie}\subseteq M_{d+k+e,\plie}$
imply
\begin{align}
0\leq l(M_{d+k+e,\plie}/M'_{r(d+k),\plie}) &=l(M_{d+k+e,\plie}/M_{d+k,\plie})-l(M'_{r(d+k),\plie}/M_{d+k,\plie}) \notag \\
&\leq l(M_{d+k+e,\plie}/M_{d+k,\plie})=e\lambda. \label{eq:cota-M'}
\end{align}
Using the additivity of $l$ and the filtration
$$M_{d,\plie}=M'_{rd,\plie}\subseteq M'_{rd+1,\plie}\subseteq\dots\subseteq
M'_{r(d+k),\plie}\subseteq M_{d+k+e,\plie}$$
we have
\begin{align*}
(k+e)\lambda=l(M_{d+k+e,\plie}/M_{d,\plie})&=l(M'_{r(d+k),\plie}/M'_{rd,\plie})+l(M_{d+k+e,\plie}/M'_{r(d+k),\plie}) \\
&=rk\,l(N'_{\plie})+l(M_{d+k+e,\plie}/M'_{r(d+k),\plie})
\end{align*}
and using (\ref{eq:cota-M'}) we have
$$\frac{\lambda}{r}=\frac{(k+e)\lambda-e\lambda}{rk}\leq l(N'_{\plie})\leq \frac{(k+e)\lambda}{rk}.$$
The lower bound for $l(N'_{\plie})$ belongs to the interval $(0,1)$, and if $k$ is
big enough so that $(k+e)/k<r/\lambda$ then the upper bound for $l(N'_{\plie})$ also
belongs to $(0,1)$. This contradicts the fact that $l(N'_{\plie})$ is an integer,
so the proof that $r$ cannot be bigger than $\lambda$ is now finished.
\end{proof}

\begin{corollary}
\label{cor:finitely-generated}
Let $M$ be a finitely generated module over $A:=\ZZ[t_1^{\pm 1},\dots,t_n^{\pm 1}]$.
Suppose that for every $1\leq i\leq n$ there exists a sequence of integers $(r_{i,j})_j$ satisfying
$r_{i,j}\to\infty$ as $j\to\infty$, and $A$-module automorphism $w_{i,j}:M\to M$
such that $w_{i,j}^{r_{i,j}}$ coincides with multiplication
by $t_i$. Then $M$ is finitely generated as a $\ZZ$-module.
\end{corollary}
\begin{proof}
Let $B=\ZZ[z_1,\dots,z_{2n}]$ and let $\phi:B\to A$ be the morphism of rings defined
by $\phi(z_{2i-1})=t_i$ and $\phi(z_{2i})=t_i^{-1}$. We can look at
$M$ as a finitely generated $B$-module via $\phi$, and the automorphisms $w_{i,j}$
in the statement are $B$-module automorphisms satisfying
$$w_{i,j}^{r_{i,j}}=\text{multiplication by $z_{2i-1}$},\qquad
(w_{i,j}^{-1})^{r_{i,j}}=\text{multiplication by $z_{2i}$}.$$
Let $B_0=\ZZ$ and $B_j=\ZZ[z_1,\dots,z_j]$ for $1\leq j\leq 2n$. Then $B_{2n}=B$,
and we can prove that $M$ is finitely generated as a $B_j$-module for any $0\leq j\leq 2n$
using descending induction on $j$, applying Theorem \ref{thm:finitely-generated}
in the induction step. It follows that $M$ is finitely generated as a $B_0$-module.
\end{proof}

\begin{corollary}
\label{cor:finitely-generated-plus}
Let $M$ be a finitely generated module over $A:=\ZZ[t_1^{\pm 1},\dots,t_n^{\pm 1}]$.
Suppose that for every $1\leq i\leq n$ there exists a nonzero integer $d_i$,
a sequence of integers $(r_{i,j})_j$ satisfying $r_{i,j}\to\infty$ as $j\to\infty$,
and $A$-module automorphisms $w_{i,j}:M\to M$ such that $w_{i,j}^{r_{i,j}}$ coincides with multiplication
by $t_i^{d_i}$. Then $M$ is finitely generated as a $\ZZ$-module.
\end{corollary}
\begin{proof}
Suppose that $M$ is generated as an $A$ module by $s_1,\dots,s_r\in M$.
Let $A'=\ZZ[t_1^{\pm d_1},\dots,t_1^{\pm d_n}]$. Then
$M$ is generated as an $A'$-module by the set
$$\{t_1^{a_1}t_2^{a_2}\dots t_n^{a_n}s_i\mid 1\leq i\leq r,\,0\leq a_i\leq |d_i|-1\text{ for each $i$}\},$$
so $M$ is finitely generated as an $A'$-module. Applying Corollary \ref{cor:finitely-generated}
to $M$ viewed as an $A'$-module, we conclude that $M$ is a finitely generated $\ZZ$-module.
\end{proof}

\section{Coverings and discrete degree of symmetry}
\label{s:coverings}

\subsection{Proof of Theorem \ref{thm:discsym-covering}}
\label{s:proof-thm:discsym-covering}
Let $X'\to X$ be a covering, where $X$ is a closed and connected
manifold. It suffices to prove
that for every natural number $b$ there is
a constant $C$, depending on $X'\to X$, such that, for every natural number $a$, if
$\Gamma_{a,b}$ acts effectively on $X$ then there is a natural
number $a'$ satisfying $Ca'\geq a$ and an effective action of
$\Gamma_{a',b}$ on $X'$.

Let $k$ be the degree of $X'\to X$.
Assume that $\Gamma:=\Gamma_{a,b}$ acts
effectively on $X$. Arguing as in \cite[\S 2.3]{M1} it follows
that there is a subgroup
$\Gamma_0\leq\Gamma$ satisfying $[\Gamma:\Gamma_0]\leq C_0$,
where $C_0$ only depends on $X$ and $k$, and an exact sequence
$$1\to F\to \Gamma_0'\stackrel{\pi}{\longrightarrow}\Gamma_0\to 1$$
where $|F|\leq k!$ and $\Gamma_0'$ acts effectively on $X'$.
By Lemma \ref{lemma:subgroup-isomorphic-to-Gamma-a-b} there is a subgroup $\Gamma_1\leq\Gamma_0$
isomorphic to $\Gamma_{a_1,b}$ for some integer $a_1$ satisfying $C_0!a_1\geq a$.
Let $\Gamma_1'=\pi^{-1}(\Gamma_1)$. By Lemma \ref{lemma:extension-subgroup-Gamma-a-b}
there is a subgroup $\Gamma''\leq\Gamma_1'$ isomorphic to $\Gamma_{a',b}$ for some
natural number $a'$ satisfying $C'a'\geq a_1$, where $C'$ only depends on $k$
(through the bound $|F|\leq k!$) and $b$. Setting $C=C_0!C'$ we have $Ca'\geq a$.
Since $\Gamma_0'$ acts effectively on $X'$, so does $\Gamma''$,
so the proof is complete.

\subsection{Proof of Theorem \ref{thm:strict-inequality}}
\label{s:strict-inequality}
Let $\Aff_{\ZZ^n}\RR^n$ denote the group of affine transformations
of $\RR^n$ that send the lattice $\ZZ^n$ to some translate of itself.
%There an exact sequence
%$$0\to\RR^n\to\Aff_{\ZZ^n}\RR^n\to\GL(n,\ZZ)\to 1.$$
%The transformations in $\Aff_{\ZZ^n}\RR^n$ descend to give affine diffeomorphisms of $T^n$.
The action of $\Aff_{\ZZ^n}\RR^n$ on $\RR^n$ descends to an action on $\RR^n/\ZZ^n=T^n$.
Denote the resulting group of transformations of $T^n$ as $\Aff T^n$.
%This coincides, as the notation suggests, with the group of affine transformations of $T^n$.
There is an exact sequence
$$0\to T^n\stackrel{\tau}{\longrightarrow}\Aff T^n\stackrel{\mu}{\longrightarrow}\GL(n,\ZZ)\to 1,$$
where $\tau$ sends $a\in T^n$ to the translation $b\mapsto a+b$.
%Using the identification $T^n=\RR^n/\ZZ^n$ we may naturally identify $H_1(T^n;\ZZ)\simeq\ZZ^n$.
The morphism $\sigma:\GL(n,\ZZ)\to\Aff T^n$ induced by the action of $\GL(n,\ZZ)$ on $\RR^n$
is a section of $\mu$.
%satisfies $\mu\circ\sigma=\Id_{\GL(n,\ZZ)}$.
If $A\in\GL(n,\ZZ)$, the action of $\sigma(A)$
on $H_1(T^n;\ZZ)$ coincides, via the natural isomorphism $H_1(T^n;\ZZ)\simeq\ZZ^n$, with $A$.
The next lemma follows from \cite[Corollary to Lemma 1]{LeeRaymond}
and \cite[Theorem 3]{LeeRaymond}.

\begin{lemma}
\label{lemma:Gamma-affine-action}
Suppose that a finite group $\Gamma$ acts effectively on $T^n$.
Let $\rho:\Gamma\to\GL(n,\ZZ)$ be the morphism given by the action of $\Gamma$ on $H_1(T^n;\ZZ)\simeq\ZZ^n$.
Then there is an embedding of groups $\eta:\Gamma\hookrightarrow\Aff T^n$
such that $\mu\circ\eta=\rho$.
%Using the isomorphism $H_1(T^n;\ZZ)\simeq\ZZ^n$ the
%action of $\Gamma$ on $H_1(T^n;\ZZ)$ gives a morphism $\rho:\Gamma\to\GL(n,\ZZ)$.
%Then there is an embedding of groups $\eta:\Gamma\hookrightarrow\Aff T^n$
%such that $\mu\circ\eta=\rho$.
\end{lemma}

Fix natural numbers $k,n$ satisfying $1\leq k\leq n-1$. Recall that
$\sigma\in\Aff T^n$ is the involution defined by
$\sigma(x_1,\dots,x_n)=(x_1+1/2,\dots,x_{k}+1/2,-x_{k+1},\dots,-x_n),$
and that
$X'=T^n$, $X=T^n/\sigma$ and $\rho:X'\to X$ denotes the projection.

Since $\Gamma_{r,n}$ acts effectively on $X'$ for every $r$, we have $\discsym X'\geq n$.
In Section \ref{s:consequences-thm:linearising-action} we proved that $\discsym X'\leq n$,
so $\discsym X'=n$.

The action of $\Gamma_{r,k}$ on $T^n$ given by
$$(\ov{a_1},\dots,\ov{a_k})\cdot (x_1,\dots,x_n)=(x_1+a_1/r,\dots,x_k+a_k/r,a_{k+1},\dots,a_n)$$
commutes with $\sigma$, and hence defines an action of $\Gamma_{r,k}$ on $X$. This action
is effective if $r$ is odd, and hence $\discsym X\geq k$.
Let us prove that $\discsym X\leq k$.
Let $T_{\sigma}^n=\{x\in T^n\mid \tau(x)\sigma=\sigma\tau(x)\}$.
Let $\iota:T^k\to T_{\sigma}^n$ be
$\iota((x_1,\dots,x_k))=(x_1,\dots,x_k,0,\dots,0)$.
Since $(x_1,\dots,x_n)\in T^n$ belongs to
$T_{\sigma}^n$ if and only if $2x_i=0$ for every $i\geq k+1$, we have
$T_{\sigma}^n/\iota(T^k)\simeq\Gamma_{2,n-k}$.
Hence, there is a short exact sequence of the form
\begin{equation}
\label{eq:Gamma-2-n-k}
0\to T^k\stackrel{\iota}{\longrightarrow}T_{\sigma}^n\to \Gamma_{2,n-k}\to 0.
\end{equation}
Suppose that $\Gamma_{r,m}$ acts effectively
on $X$. The arguments in the proof of Theorem \ref{thm:discsym-covering} imply the existence
of a constant $C'$ depending only on $k$ and $n$, a subgroup
$\Gamma_0\leq\Gamma_{r,m}$ satisfying $[\Gamma_{r,m}:\Gamma_0]\leq C'$, and a central extension
of groups
$$1\to Z\to \Gamma_0'\to\Gamma_0\to 1$$
such that $\Gamma_0'$ acts effectively on $T^n$, and $Z=\{\Id,\sigma\}$.
So the order of every element of $\Gamma_0'$ is smaller than or equal to $2r$.
Let $\rho:\Gamma_0'\to\GL(n,\ZZ)$
be the morphism induced by the action of $\Gamma_0'$ on $H_1(T^n;\ZZ)\simeq\ZZ^n$.
By Lemma \ref{lemma:Gamma-affine-action} there is a monomorphism
$\eta:\Gamma_0'\to\Aff T^n$ satisfying $\rho=\mu\circ\eta$.
By Theorem \ref{thm:Minkowski}
$\Gamma_0'':=\Ker \rho$ satisfies
$[\Gamma_0':\Gamma_0'']\leq C$ for some $C$ depending only on $n$.
Then $\eta(\Gamma_0'')\leq \tau(T_{\sigma}^n)$, so by (\ref{eq:Gamma-2-n-k})
there is a subgroup $\Gamma_0'''\leq\Gamma_0''$ satisfying $[\Gamma_0'':\Gamma_0''']\leq 2^{n-k}$
and an embedding $\Gamma_0'''\hookrightarrow T^k$. By
Lemma \ref{lemma:upper-bound-subgroup-T-d} we have $|\Gamma_0'''|\leq (2r)^k$ because $\Gamma_0'''\leq\Gamma_0'$, so
$$r^m=|\Gamma_{r,m}|\leq C'|\Gamma_0|=
\frac{C'}{2}|\Gamma_0'|\leq \frac{C'}{2}C2^{n-k}(2r)^k=C'C2^{n-1}r^k.$$
Consequently, if $r>C'C2^{n-1}$ then $m\leq k$.

\section{Finite generation of the homology of abelian covers}
\label{s:cohomology-finitely-generated}

Denote by $\pi:\RR^k\to T^k=\RR^k/\ZZ^k$ the quotient map.
Given a topological space $X$
and a continuous map $\phi:X\to T^k$ we denote by
$$X_{\phi}=\{(x,u)\in X\times\RR^k\mid \phi(x)=\pi(u)\}$$
the pullback to $X$ of the covering $\RR^k\to T^k$. The projection
$$\rho_{\phi}:X_{\phi}\to X,\qquad \rho_{\phi}(x,u)=x$$
is an unramified covering map. We can also look at $\rho_{\phi}:X_{\phi}\to X$
as a principal $\ZZ^k$-bundle, where $\ZZ^k$ acts on $X_{\phi}$
as follows: if $\nu\in \ZZ^k$ and $(x,u)\in X_{\phi}$
then $\nu\cdot (x,s)=(x,u+\nu)$.
Hence, $X_{\phi}$ is an abelian cover.
Standard results on fiber bundles imply the following.

\begin{lemma}
\label{lemma:homotopic-bundles-isomorphic}
Suppose that $X$ is paracompact.
If two continuous maps $\phi,\psi:X\to T^k$ are homotopic then there
is a $\ZZ^k$-equivariant homeomorphism $\zeta:X_{\phi}\to X_{\psi}$
such that $\rho_{\phi}=\rho_{\psi}\circ\zeta$.
\end{lemma}

We identify the group ring $\ZZ[\ZZ^k]$ with the additive group
of finitely supported functions $\ZZ^k\to\ZZ$ with ring structure given by convolution.
Let $e_1,\dots,e_k$ denote the canonical basis of $\ZZ^k$,
and let $t_i\in\ZZ[\ZZ^k]$ denote the characteristic function of $\{e_i\}\subset\ZZ^k$.
Then $\ZZ[\ZZ^k]\simeq \ZZ[t_1^{\pm 1},\dots,t_k^{\pm 1}]$.
The action of $\ZZ^k$ on $X_{\phi}$ induces an action on $H_*(X_{\phi};\ZZ)$
or, equivalently, a structure on $H_*(X_{\phi};\ZZ)$
of module over the group ring $\ZZ[\ZZ^k]\simeq\ZZ[t_1^{\pm 1},\dots,t_k^{\pm 1}]$.

\begin{lemma}
\label{lemma:finite-generated-cohomology}
Let $X$ be a closed topological manifold, and let $\phi:X\to T^k$ be a
continuous map. Then $H_*(X_{\phi};\ZZ)$ is finite generated
as a $\ZZ[t_1^{\pm 1},\dots,t_k^{\pm 1}]$-module.
\end{lemma}
\begin{proof}
Since compact topological manifolds are Euclidean Neighborhood Retracts
(see e.g. \cite[Corollary A.9]{Hatcher}) we can identify homeomorphically $X$ with a closed subset of some
Euclidean space $\RR^N$
in such a way that there exists an open subset $\oO\subset\RR^N$ containing $X$ and a retraction
$r:\oO\to X$. For any $x\in X$ let $B_x\subset\oO$ be an open ball centered at $x$.
By compactness we may choose a finite set of points $x_1,\dots,x_s\in X$ such that
$X\subset Y:=B_{x_1}\cup\dots\cup B_{x_s}$. Let $B_i=B_{x_i}$ and let $\psi=\phi\circ r:Y\to T^k$.

Let $A:=\ZZ[t_1^{\pm 1},\dots,t_k^{\pm 1}]$.
For every $i$ the space $B_i$ is contractible, so by Lemma \ref{lemma:homotopic-bundles-isomorphic}
the principal $\ZZ^k$-bundle $(B_i)_{\psi}\to B_i$ is trivial. Hence, for every subspace $S\subseteq B_i$
we have an isomorphism of $A$-modules
$H_*(S_{\psi};\ZZ)\simeq H_*(S;\ZZ)\otimes_{\ZZ}A$. So if $S\subseteq B_i$ has the property that
$H_*(S;\ZZ)$ is a finitely generated abelian group,
then $H_*(S_{\psi};\ZZ)$ is a finitely generated $A$-module.
It also follows that
$H_*((B_i)_{\psi};\ZZ)$ is a free
$A$-module of rank $1$.

%For any $1\leq j\leq s$ we
Let $B_{\leq j}=B_1\cup\dots\cup B_j$.
We next prove that $H_*((B_{\leq j})_{\psi};\ZZ)$ is a finitely generated $A$-module for
every $j$, using ascending induction on $j$.
The case $j=1$ has been already been proved. Suppose that $j>1$ and that the claim is true for $j-1$.  The Mayer--Vietoris exact sequence (MVES)
\begin{multline*}
\dots\to
H_k((B_{\leq j-1})_{\psi}\cap (B_j)_{\psi};\ZZ)
\to
H_k((B_{\leq j-1})_{\psi};\ZZ)
\oplus H_k((B_j)_{\psi};\ZZ)
\to \\
\to
H_k((B_{\leq j})_{\psi};\ZZ)
\to
H_{k-1}((B_{\leq j-1})_{\psi}\cap (B_j)_{\psi};\ZZ)
\to \dots
\end{multline*}
is an exact sequence of $A$-modules, by the naturality of the MVES
and the fact that the $A$-module structure
on each term comes from an action of $\ZZ^k$ on the spaces commuting with the inclusions
$$(B_{\leq j-1})_{\psi}\hookleftarrow (B_{\leq j-1})_{\psi}\cap (B_j)_{\psi}\hookrightarrow (B_j)_{\psi}$$
and
$$(B_{\leq j-1})_{\psi}\hookrightarrow (B_{\leq j})_{\psi} \hookleftarrow (B_j)_{\psi}.$$
By the induction hypothesis
$H_k((B_{\leq j-1})_{\psi};\ZZ)
\oplus H_k(B_j)_{\psi};\ZZ)$ is a finitely generated $A$-module.
We have
$(B_{\leq j-1})_{\psi}\cap (B_j)_{\psi}=(B_{\leq j-1}\cap B_j)_{\psi}$,
$B_{\leq j-1}\cap B_j$ is obviously a subset of $B_j$, and
$H_*(B_{\leq j-1}\cap B_j;\ZZ)$ is finitely generated (because
$B_{\leq j-1}\cap B_j$ is the union of finitely many convex subsets of $\RR^N$).
Hence, $H_k((B_{\leq j-1})_{\psi}\cap (B_j)_{\psi};\ZZ)$
is a finitely generated $A$-module.

Since $A$ is Noetherian, the previous considerations and the exactness of the sequence
imply that $H_k((B_{\leq j})_{\psi};\ZZ)$ is a finitely generated $A$-module for every $k$.
Finally, $(B_{\leq j})_{\psi}$ is an $N$-dimensional topological manifold,
so its homology vanishes in dimensions bigger than $N$. It follows that
the entire homology $H_*((B_{\leq j})_{\psi};\ZZ)$ is a finitely
generated $A$-module, so the proof of the claim is complete.

To conclude the proof note that the inclusion $\iota:X\hookrightarrow Y$ and the retraction
$r:Y\to X$ induce $\ZZ^k$-equivariant maps $\iota':X_{\phi}\hookrightarrow Y_{\psi}$
and $r':Y_{\psi}\to X_{\phi}$ satisfying $r'\circ\iota'=\Id_{X_{\phi}}$.
It follows that $H_*(X_{\phi};\ZZ)$ is an $A$-submodule of $H_*(Y_{\psi};\ZZ)$. Since $A$
is Noetherian and $H_*(Y_{\psi};\ZZ)$ is finitely generated, it follows that
$H_*(X_{\phi};\ZZ)$ is also finitely generated.
\end{proof}

\section{Proofs of Theorem \ref{thm:discsym-rationally-hypertoral}
and Corollary \ref{cor:homeomorphic-torus}}
\label{s:proof-thm:main-thm}

Statement (1) of Theorem \ref{thm:discsym-rationally-hypertoral}
was proved in Section \ref{s:consequences-thm:linearising-action},
so we only need to prove (2).
Let $X$ be a rationally hypertoral $n$-dimensional manifold satisfying
$\discsym (X)=n$, so that $X$ supports effective actions
of $\Gamma_{r,n}=(\ZZ/r)^n$ for arbitrarily large integers $r$. Fix a continuous map $\phi:X\to T^n$
of nonzero degree. Let
$$d=|\deg \phi|.$$
Define the principal $\ZZ^n$-bundle $X_{\phi}\to X$ as in the previous
section. As we explained, $H_*(X_{\phi};\ZZ)$ has a structure of
module over $\ZZ[\ZZ^n]$ and, by Lemma \ref{lemma:finite-generated-cohomology},
$H_*(X_{\phi};\ZZ)$ is finitely generated as a $\ZZ[\ZZ^n]$-module.
Recall that $\ZZ[\ZZ^n]\simeq\ZZ[t_1^{\pm 1},\dots,t_n^{\pm 1}]$, where
$t_i$ is the characteristic function of the $i$-th element of the canonical basis of $\ZZ^n$.

The following lemma describes how certain homeomorphisms of $X$ lift to homeomorphisms of $X_{\phi}$. Recall that for any $a\in T^m$ we denote by $\tau(a):T^m\to T^m$ the translation $\tau(a)(t)=t+a$.

\begin{lemma}
\label{lemma:lifting-to-X-psi}
Let $\psi:X\to T^m$ be a map.
Let $f:X\to X$ be a homeomorphism of order $r$ satisfying $\tau(a)\circ \psi=\psi\circ f$ for some
$a\in T^m$, which necessarily satisfies $ra=0$.
Let $v\in \RR^m$ satisfy $\pi(v)=a$. There exist a lift of $f$, $g:X_{\psi}\to X_{\psi}$, satisfying $g^r(x,u)=(x,u+rv)$ for every $(x,u)\in X_{\psi}$. Furthermore, $g$ commutes
with the action of $\ZZ^m$ on $X_{\psi}$.
\end{lemma}
\begin{proof}
Recall that $X_{\psi}=\{(x,u)\in X\times\RR^m\mid \psi(x)=\pi(u)\}$.
Define $g:X_{\psi}\to X_{\psi}$ by $g(x,u)=(f(x),u+v)$.
The equality $\tau(a)\circ \psi=\psi\circ f$
guarantees that this is indeed a well defined homeomorphism of $X_{\psi}$.
It is immediate that $g^r(x,u)=(x,u+rv)$ for every $(x,u)\in X_{\psi}$ and that $g$ commutes with the action of $\ZZ^n$.
\end{proof}

\begin{lemma}
\label{lemma:roots-of-t-i}
For every $1\leq j\leq n$ there exists a nonzero integer $d_j$
and a sequence of natural numbers $o_{i,j}$ satisfying $o_{i,j}\to\infty$
as $i\to\infty$, and isomorphisms of $\ZZ[\ZZ^n]$-modules $w_{i,j}:H_*(X_{\phi};\ZZ)\to H_*(X_{\phi};\ZZ)$,
such that $w_{i,j}^{o_{i,j}}$ coincides with multiplication by $t_j^{d_j}$.
\end{lemma}
\begin{proof}
Let $C$ be the number given by applying Theorem \ref{thm:Minkowski} to $X$.
%there exists a number $C$ such that for every action of a finite group $G$ on $X$
%the kernel $G'$ of the natural map $G\to \Aut H^1(X;\ZZ)$ satisfies $[G:G']\leq C$.
Let $0<r_1<r_2<\dots $ be the infinite sequence of integers such that
$X$ supports an effective action of $G_i:=\Gamma_{r_i,n}$ for every $i$.
Then $G_i':=\Ker(G_i\to \Aut H^1(X;\ZZ))$ satisfies
$[G_i:G_i']\leq C$. By Lemma \ref{lemma:subgroup-isomorphic-to-Gamma-a-b}
there is a subgroup $G_i''\leq G_i'$ such that $G_i''\simeq\Gamma_{s_i,n}$ for a natural
number $s_i$ satisfying $C!s_i\geq r_i$. In particular, $s_i\to\infty$ as $i\to\infty$.

Applying Theorem \ref{thm:linearising-action} to the action of $G_i''$ on $X$ we obtain
a morphism of groups $$\eta_i:G_i''\to T^n$$ and an $\eta_i$-equivariant map $\psi_i:X\to T^n$
which is homotopic to $\phi$. Also, $|\Ker\eta_i|$ divides $d$, so
$|\eta(G_i'')|\geq s_i^n/d$. For every $1\leq j\leq n$ let
$S_j=\{(\theta_1,\dots,\theta_n)\in T^n\mid \theta_k=0\text{ for $k\neq j$}\}$.

By Lemma \ref{lemma:upper-bound-intersection-with-S} we have
$$\ov{o}_{i,j}:=|\eta(G_i'')\cap S_j|\geq \frac{s_i^n}{d\cdot s_i^{n-1}}=\frac{s_i}{d}.$$
Recall that $\pi:\RR^n\to\RR^n/\ZZ^n=T^n$ is the quotient map, and that
$e_j\in\RR^n$ denotes the $j$-th element of the canonical basis.
The element $e_{i,j}=\pi(e_j/\ov{o}_{i,j})\in T^n$ is a generator of $\eta(G_i'')\cap S_j$.
Let
$$f_{i,j}:X\to X$$
be the homeomorphism given by the action of an element of $\eta^{-1}(e_{i,j})\subseteq G_i''$.
We have
$\psi_i\circ f_{i,j}=\tau(e_j/\ov{o}_{i,j})\circ \psi_i$.
The order of $f_{i,j}$ is $o_{i,j}=d_{i,j}\ov{o}_{i,j}$,
where $d_{i,j}$ is a natural number dividing $d$.
Passing to a subsequence and relabelling accordingly
we may assume that all natural numbers $d_{1,j},d_{2,j},\dots$ are equal to the same number $d_j$.
By Lemma \ref{lemma:lifting-to-X-psi}
there is a homeomorphism $g_{i,j}:X_{\psi_i}\to X_{\psi_i}$ such that
$g_{i,j}^{o_{i,j}}:X_{\psi}\to X_{\psi}$ coincides with the action of $d_je_j$ on $X_\psi$
given by the structure of principal $\ZZ^n$-bundle on $X_{\psi}$.

Since $\psi_i$ is homotopic to $\phi$, by Lemma \ref{lemma:homotopic-bundles-isomorphic} there
is a $\ZZ^n$-equivariant homeomorphism $\zeta_i:X_{\phi}\to X_{\psi_i}$.
Let $w_{i,j}:H_*(X_{\phi};\ZZ)\to H_*(X_{\phi};\ZZ)$ be the isomorphism induced by the
homeomorphism
$\zeta_i^{-1}\circ g_{i,j}\circ\zeta_i:X_{\phi}\to X_{\phi}.$
Then $w_{i,j}^{o_{i,j}}$ coincides with multiplication by $t_j^{d_j}$.
Since ${\ov o}_{i,j}\geq s_i/d$ and $s_i\to \infty$ as $i\to\infty$, we conclude that
$o_{i,j}\to\infty$ as $i\to\infty$.
\end{proof}

Combining the previous lemma with Corollary \ref{cor:finitely-generated-plus},
it follows that $H_*(X_{\phi};\ZZ)$ is a finitely generated $\ZZ$-module.
%(Note that we may apply Corollary \ref{cor:finitely-generated-plus}
%to the $\ZZ[t_1^{\pm 1},\dots,t_n^{\pm 1}]$-module
%$H_*(X_{\phi};\ZZ)$ thanks to Lemma \ref{lemma:finite-generated-cohomology}.)

\begin{lemma}
We have $H_k(X_{\phi};\ZZ)=0$ for every $k>0$.
\end{lemma}
\begin{proof}
Suppose that $H_k(X_{\phi};\ZZ)\neq 0$ for some $k>0$. By the universal coefficient
theorem, there exists some prime $p$ such that $H_k(X_{\phi};\ZZ/p)\neq 0$ for some $k>0$.
The action of $\ZZ^n$ on $X_{\phi}$ induces a morphism $\alpha_p:\ZZ^n\to\Aut(H_*(X_{\phi};\ZZ/p))$.
Since $H_*(X_{\phi};\ZZ)$ is a finitely generated $\ZZ$-module,
$H_*(X_{\phi};\ZZ/p)$ is a finite group,
so $\Lambda=\Ker \alpha_p$ %(\ZZ^n\to\Aut(H_*(X_{\phi};\ZZ/p)))$$
has finite index in $\ZZ^n$.

Consider the action of $\Lambda$ on $X_{\phi}\times\RR^n$ given by
$\lambda\cdot((x,u),v)=(\lambda\cdot (x,u),v-\lambda)=((x,u+\lambda),v-\lambda)$,
and let $X_{\phi}\times_{\Lambda}\RR^n$ denote the quotient space.
We have maps
$$\xymatrix{ \RR^n/\Lambda & X_{\phi}\times_{\Lambda}\RR^n\ar[l]_-{\Pi}\ar[r]^-{\Theta} & X_{\phi}/\Lambda},$$
where $\Pi$ is induced by the projection $X_{\phi}\times\RR^n\to\RR^n$
and $\Theta$ is induced by the projection $X_{\phi}\times\RR^n\to X_{\phi}$.
The map $\Pi$ is a fibration with fiber $X_{\phi}$. The map $\Theta$ is a fibration
with fiber $\RR^n$ and hence is a homotopy equivalence. (This is of course a
general phenomenon: $X_{\phi}\times_{\Lambda}\RR^n$ is the Borel construction for
the action of $\Lambda$ on $X_{\phi}$, and the fact that $X_{\phi}\times_{\Lambda}\RR^n$
is homotopy equivalent to the quotient $X_{\phi}/\Lambda$ is a consequence of the fact
that $\Lambda$ acts freely on $X_{\phi}$.)
Since $X_{\phi}/\Lambda$ is an $n$-dimensional manifold, we have
\begin{equation}
\label{eq:cohomologia-acotada}
H_k(X_{\phi}\times_{\Lambda}\RR^n;\ZZ/p)=H_k(X_{\phi}/\Lambda;\ZZ/p)=0
\qquad\text{for every $k>n$}.
\end{equation}

Since $\Lambda$ acts trivially $H_*(X_{\phi};\ZZ/p)$, the monodromy action of $\pi_1(\RR^n/\Lambda)\simeq\Lambda$
on the homology with $\ZZ/p$-coefficients of the fibers of $\Pi$ is trivial.
Consequently,
the homology Serre spectral sequence for the fibration $\Pi$ takes the form
$$H_a(\RR^n/\Lambda;H_b(X_{\phi};\ZZ/p))\simeq
H_a(\RR^n/\Lambda;\ZZ/p)\otimes_{\ZZ/p} H_b(X_{\phi};\ZZ/p)\Longrightarrow H_{a+b}(X_{\phi}\times_{\Lambda}\RR^n;\ZZ/p).$$

Let $l=\max\{k\mid H_k(X_{\phi};\ZZ/p)\neq 0\}$. By our choice of $p$, we have $l>0$. Then
$H_n(\RR^n/\Lambda;\ZZ/p)\otimes_{\ZZ/p} H_l(X_{\phi};\ZZ/p)$ is a nonzero entry in the second page
of the spectral sequence, and for dimension reasons it is contained in the kernel of every
differential %in the spectral sequence,
and none of its elements is killed by any differential;
consequently, $H_n(\RR^n/\Lambda;\ZZ/p)\otimes_{\ZZ/p} H_l(X_{\phi};\ZZ/p)$ can be identified
with a subquotient of $H_{n+l}(X_{\phi}\times_{\Lambda}\RR^n;\ZZ/p)$. Hence (\ref{eq:cohomologia-acotada}) implies that $l=0$,
which is a contradiction.
\end{proof}

Since $H_0(X_{\phi};\ZZ)$ is finitely generated, $\pi_0(X_{\phi})$
is finite. Since $X\simeq X_{\phi}\times_{\ZZ^n}\RR^n$,
the space $X_{\phi}\times_{\ZZ^n}\RR^n$ is connected.
The projection $\Pi:X_{\phi}\times_{\ZZ^n}\RR^n\to\RR^n/\ZZ^n$ is a fibration with fiber $X_{\phi}$, so
the monodromy action of $\ZZ^n$ on $\pi_0(X_{\phi})$ is transitive. The monodromy action
coincides with the action naturally induced by the action of $\ZZ^n$ on $X_{\phi}$.
So we have proved the following lemma.

\begin{lemma}
$\pi_0(X_{\phi})$ is finite, and the action of $\ZZ^n$ on $X_{\phi}$ induces a transitive
action on $\pi_0(X_{\phi})$.
\end{lemma}

Fix an arcconnected component $X_{\phi}^0\subseteq X_{\phi}$.
Since $H_k(X_{\phi};\ZZ)=0$ for every $>0$, $X_{\phi}^0$ is acyclic.
%has the same integral homology as the point:
%\begin{equation}
%\label{eq:cohomologia-X-phi-0}
%H_0(X_{\phi}^0;\ZZ)\simeq\ZZ\qquad\text{and}\qquad H_k(X_{\phi}^0;\ZZ)=0\text{ for $k>0$}.
%\end{equation}
Let $V\leq\ZZ^n$ be the subgroup consisting of those elements of $\ZZ^n$
whose action of $X_{\phi}$ maps $X_{\phi}^0$ to itself. By the previous lemma,
$V$ has finite index in $\ZZ^n$ and we have $X=X_{\phi}/\ZZ^n=X_{\phi}^0/V$.
Since $H_1(X_{\phi}^0)=0$, it follows that $X_{\phi}^0$ is isomorphic to the
universal abelian cover of $X$.

Since $V$ has finite index in $\ZZ^n$, the quotient $\RR^n/V$ is homeomorphic to
$T^n$. Arguing as in the definition of $X_{\phi}\times_\Lambda\RR^n$
we prove that $X=X_{\phi}^0/V$ is homotopy equivalent
to $X_{\phi}^0\times_V\RR^n$, where the latter is defined exactly as $X_{\phi}\times_\Lambda\RR^n$
but replacing $X_{\phi}$ resp. $\Lambda$ with $X_{\phi}^0$ resp. $V$. The fibers of the
projection $Q:X_{\phi}^0\times_V\RR^n\to\RR^n/V$
can be identified with $X_{\phi}^0$, and hence are acyclic; it follows that
$Q$ induces an isomorphism in integral homology.
So $H_*(X;\ZZ)=H_*(X_{\phi}^0/V;\ZZ)$ is isomorphic to $H_*(\RR^n/V;\ZZ)\simeq H_*(T^n;\ZZ)$. This finishes the proof of statement (2) of Theorem \ref{thm:discsym-rationally-hypertoral}.

We next prove Corollary \ref{cor:homeomorphic-torus}, using
the notation of the previous arguments. Assume first that $\pi_1(X)$ is solvable.
The projection $X_{\phi}^0\to X$ is a covering space, so it identifies $\pi_1(X_{\phi}^0)$ with
a subgroup of $\pi_1(X)$. Hence, $\pi_1(X_{\phi}^0)$ is solvable. Since $X_\phi^0$ is acyclic,
the abelianization of $\pi_1(X_{\phi}^0)$ is trivial, which implies that $\pi_1(X_{\phi}^0)$
itself is trivial. By Hurewicz's theorem, $X_{\phi}^0$ is contractible
(see e.g. \cite[Corollary 4.33]{Hatcher}).

\begin{remark}
\label{rmk:solvable}
This is the only point where we use that $\pi_1(X)$ is (virtually) solvable.
Note that there exist non-contractible acyclic manifolds.
Indeed, by a result of Kervaire \cite[Theorem 1]{Kervaire} any finitely generated acyclic group
is the fundamental group of an integral smooth homology sphere of dimension $n>4$,
and there are plenty of examples of such groups (see e.g. \cite{Berrick}).
Removing a point from such manifold we obtain an acyclic
manifold with the same fundamental group. It is not clear, however, whether a non simply connected
acyclic $n$-dimensional manifold can support a free action of $\ZZ^n$ with compact quotient
(let alone that its quotient by $\ZZ^n$
supports free actions of $(\ZZ/r)^n$ for arbitrarily large $r$). So it could be the case
that our assumption that $\pi_1(X)$ is virtually solvable is unnecessary.
\end{remark}

It follows that $Q:X_{\phi}^0\times_V\RR^n\to\RR^n/V$ is a homotopy equivalence. Precomposing it with
a homotopy inverse of the projection $X_{\phi}^0\times_V\RR^n\to X_{\phi}^0/V=X$ we obtain a
homotopy equivalence $X\to\RR^n/V$. By topological rigidity of tori,
$X$ is homeomorphic to $T^n$.

To conclude, let us prove Corollary \ref{cor:homeomorphic-torus}
in the general case.
Suppose that $X$ is a rationally hypertoral manifold satisfying $\discsym(X)=n$
and that $\pi_1(X)$ is virtually solvable.
Then there is a finite covering $r:X'\to X$ such that $\pi_1(X')$ is solvable.
Let $\phi:X\to T^n$ be
a map of nonzero degree and let $\phi'=\phi\circ r$.
Then $\deg\phi'=\deg\phi\cdot\deg r\neq 0$ and we have a Cartesian diagram
$$\xymatrix{X'_{\phi'}\ar[r]^{r_{\phi}}\ar[d]_{\rho_{\phi'}} & X_{\phi}\ar[d]^{\rho_{\phi}} \\
X'\ar[r]^r  & X.}$$
In particular, $r_{\phi}$ is a finite (unramified) covering space.
By Theorem \ref{thm:discsym-covering} we have $\discsym(X')\geq n$. Applying the previous
discussion to $X'$ we conclude that the connected components of $X'_{\phi'}$
are contractible. Let $X_\phi^0$ be any connected component of $X_{\phi}$ and
denote by $(X'_{\phi'})^0$ its preimage under $r_{\phi}$. Let $\pi=\pi_1(X_{\phi}^0)$, so that
$X_{\phi}^0=(X'_{\phi'})^0/\pi$, where $\pi$ acts freely on $(X'_{\phi'})^0$. Note that $\pi$
is finite, because $r_{\phi}$ is a finite covering. The freeness
of the action implies that $\pi$ acts freely on the set of connected components of
$(X'_{\phi'})^0$ because by Smith's theory
a homeomorphism of primer order
of a contractible manifold has necessarily some fixed point
(see e.g. \cite[Chap III, Corollary 4.6]{Bo}). Hence, $X_{\phi}^0$
is also contractible, so the same argument as in the case of solvable fundamental group
allows to prove that $X$ is homeomorphic to $T^n$.

\section{Proof of Theorem \ref{thm:bounding-discsym}}
\label{s:proof-thm:bounding-discsym}

%Statement (1) of Theorem \ref{thm:bounding-discsym} follows from combining
%Theorem \ref{thm:Minkowski},
%Theorem \ref{thm:linearising-action} and Lemma \ref{lemma:upper-bound-subgroup-T-d}.
%For the proof of statement (2)
We will need the following two lemmas.

\begin{lemma}
\label{lemma:hypertoral-product}
Let $X,Y$ be two closed connected topological manifolds.
Then $X\times Y$ is rationally hypertoral if and only if both $X$ and $Y$ are
rationally hypertoral.
\end{lemma}
\begin{proof}
Let $\pi_X,\pi_Y$ be the projections from $X\times Y$ to $X,Y$ respectively.
Let $m=\dim X$ and $n=\dim Y$.
If $X$ and $Y$ are rationally hypertoral then there exist classes
$\alpha_1,\dots,\alpha_m\in H^1(X;\ZZ)$ and $\beta_1,\dots,\beta_n\in H^1(Y;\ZZ)$
such that $\alpha_1\smile\dots\smile\alpha_m\neq 0$ and $\beta_1\smile\dots\smile\beta_n\neq 0$.
Then $\pi_X^*\alpha_1\smile\dots\smile\pi_X^*\alpha_m\smile \pi_Y^*\beta_1\smile\dots\smile\pi_Y^*\beta_n\neq 0$ by K\"unneth,
%because the map $H^m(X;\ZZ)\otimes H^n(Y;\ZZ)\to H^{m+n}(X\times Y;\ZZ)$ given by
%pullback and cup-product is an isomorphism,
so $X\times Y$ is rationally hypertoral. Conversely, suppose that there exist classes
$\gamma_1,\dots,\gamma_{m+n}\in H^1(X\times Y;\ZZ)$ such that $\gamma_1\smile\dots\smile\gamma_{m+n}\neq 0$.
By K\"unneth we may write $\gamma_i=\delta_i+\epsilon_i$ where $\delta_i\in\pi_X^*H^1(X;\ZZ)$ and
$\epsilon_i\in\pi_Y^*H^1(Y;\ZZ)$. Expanding the product
$\gamma_1\smile\dots\smile\gamma_{m+n}=(\delta_1+\epsilon_1)\smile\dots\smile(\delta_{m+n}+\epsilon_{m+n})$
the only terms that are not automatically zero are equal (after reordering terms) to terms of the form
$\delta_{i_1}\smile\dots\smile\delta_{i_m}\smile\epsilon_{j_1}\smile\dots\smile\epsilon_{j_n}$
where $\{i_1,\dots,i_m\}\cup\{j_1,\dots,j_n\}=\{1,\dots,m+n\}$. Some of these monomials
is different from zero, and hence both $X$ and $Y$ are rationally hypertoral.
\end{proof}

\begin{lemma}
\label{lemma:hypertoral-connected-sum}
Let $X,Y$ be closed connected topological manifolds of the same dimension. Suppose that
$X\sharp Y$ is rationally hypertoral. Then at least one of the manifolds
$X,Y$ is also rationally hypertoral.
\end{lemma}
\begin{proof}
If $n=1$ the statement is obvious, and if $n=2$ it follows from the classification
of closed connected surfaces. Suppose that $n\geq 3$ and that
$Z:=X\sharp Y$ is rationally hypertoral.
Let $X_0,Y_0$ be the complementaries of open balls in $X,Y$ respectively, so that
$\partial X_0\cong S^{n-1}\cong \partial Y_0$.
We identify $Z$ with $X_0\cup _{\partial X_0\cong \partial Y_0}Y_0$.
Denote by $i_X:X_0\hookrightarrow Z$
and $i_Y:Y_0\hookrightarrow Z$ the natural inclusions.

Let $W:=i_X(X_0)\cap i_Y(Y_0)$, choose a homeomorphism $\xi:W\to S^{n-1}$, and let $E$
be the result of attaching to $Z$ an $n$-disk $D^n$ along its boundary $\partial D^n=S^{n-1}$,
using $\xi$. Let $\phi:Z\to T^n$ be a map of nonzero degree.
Since $\pi_{n
-1}(T^n)=0$, the restriction of $\phi$ to $W$ is homotopically trivial, so $\phi$ extends
to a continuous map $\psi:E\to T^n$. By contracting $D^n\subset E$ we obtain a map $c:E\to X\vee Y$
that is a homotopy equivalence. Composing $\psi$ with a homotopy inverse of $c$ we obtain a map
$\zeta:X\vee Y\to T^n$. The map $\phi$ coincides up to homotopy with the composition
$\zeta\circ(c|_Z):Z\longrightarrow X\vee Y\longrightarrow T^n.$
Since the degree of $\phi$ is nonzero, it follows that the restriction of $\zeta$ to one of the
two summands $X\subset X\vee Y$ or $Y\subset X\vee Y$ has to be nonzero, so the corresponding manifold
is rationally hypertoral.
\end{proof}

We now prove Theorem \ref{thm:bounding-discsym}.
Let $X=(Y\sharp Y')\times Z$ be a rationally hypertoral manifold, where
$\dim Y=\dim Y'=k$ and $\dim Z=n-k$ for some integer $0\leq k<n$. Since $X$ is orientable,
$Y,Y',Z'$ are orientable. By Lemma \ref{lemma:hypertoral-product}
both $Y\sharp Y'$ and $Z$ are hypertoral, and by Lemma \ref{lemma:hypertoral-connected-sum}
at least one of the manifolds $Y,Y'$ (say, $Y$) is rationally hypertoral.
Assume that $H^*(Y';\ZZ)$ is not isomorphic to $H^*(S^{n-k};\ZZ)$.
Choose maps of nonzero degree $\phi_Y:Y\to T^{k}$, $\phi_Z:Z\to T^{n-k}$ and let
$\phi=(\phi_Y',\phi_Z):X=(Y\sharp Y')\times Z\to T^{k}\times T^{n-k}=T^n$, where
$\phi_Y'=\phi_Y\circ c_{Y'}:Y\sharp Y'\to T^{n-k}$ and
$c_{Y'}:Y\sharp Y'\to Y$ is the map collapsing $Y'$. Then $d:=\deg \phi$ is nonzero.

We prove that $\discsym(X)\leq n-k$ by contradiction. Assume that
there exists a sequence of natural numbers $r_i\to\infty$ such that $\Gamma_i:=\Gamma_{r_i,n-k+1}$
acts effectively on $X$. Arguing as in the end of Section \ref{ss:proof-thm:weak-bound-disc-sym},
we may assume without loss of generality that, for each $i$, $r_i=p_i^{e_i}$ for some prime $p_i$ and a natural number $e_i$.

By Theorem \ref{thm:Minkowski}, Theorem \ref{thm:linearising-action}
and Lemma \ref{lemma:subgroup-isomorphic-to-Gamma-a-b}, there exists a constant $C$
and, for every $i$, a subgroup $\Gamma_i'\leq\Gamma_i$ isomorphic to $\Gamma_{r_i',n-k+1}$, with
$r_i'$ dividing $r_i$ and satisfying $Cr_i'\geq r_i$, a morphism $\eta_i:\Gamma_i'\to T^n$
satisfying $|\Ker\eta_i|\leq d$, and an $\eta_i$-equivariant map
$$\phi_i:X\to T^n$$ homotopic to $\phi$.
By Lemma \ref{lemma:subgroup-isomorphic-to-Gamma-a-b-quotient} applied to the map
$\eta_i:\Gamma_i'\to\eta_i(\Gamma_i')$, there is a constant $C'$ and, for each $i$, a subgroup
$\Gamma_i''\leq \eta_i(\Gamma_i')$ isomorphic to $\Gamma_{s_i,n-k+1}$ with $C's_i\geq r_i'$.

We identify $T^k$ with the subgroup of $T^n$ consisting of elements
whose last $n-k$ coordinates vanish. Let $Q:T^n\to T^n/T^{k}\simeq T^{n-k}$ be the quotient map.
We claim that there exists an element $\gamma_i\in T^k\cap\Gamma_i''$ of order $s_i$.
Otherwise, $T^k\cap\Gamma_i''$ would be included in $p_i\Gamma_i''$. Then
$Q(\Gamma_i'')\simeq \Gamma_i''/(T^k\cap\Gamma_i'')$ would have a quotient
isomorphic to $\Gamma_i''/p_i\Gamma_i''\simeq \Gamma_{p_i,n-k+1}$.
By Lemma \ref{lemma:surjection-subgroup}, $Q(\Gamma_i'')$ would have a subgroup isomorphic to
$\Gamma_{p_i,n-k+1}$, which contradicts Lemma \ref{lemma:upper-bound-subgroup-T-d}.
Let $\pi_j:T^k\to S^1$ denote the projection to the $j$-th factor. Since $s_i$ is a
prime power, for each $i$ there exists some $j_i\in\{1,\dots,k\}$ such that
the order of $\pi_{j_i}(\gamma_i)$ is equal to the order of $\gamma_i$, i.e., to $s_i$.
Passing to a subsequence we may assume that all $j_i$ are equal to some
$j\in\{1,\dots,k\}$.

Define maps $X\to S^1$ by $\psi_i=\pi_j\circ\phi_i$ and $\zeta=\pi_j\circ\phi$.
Then $X_{\psi_i}\to X$ is a $\ZZ$-principal bundle and $H_*(X_{\psi_i};\ZZ)$ is a finitely generated
$\ZZ[\ZZ]=\ZZ[t^{\pm 1}]$ module by Theorem \ref{thm:finitely-generated}.
The maps $\psi_i$ and $\zeta$ are homotopic,
so there exists a $\ZZ$-equivariant homeomorphism $\chi_i:X_{\zeta}\to X_{\psi_i}$.

%We are next going to prove that $H_*(X_{\zeta};\ZZ)$ is finitely generated as a $\ZZ$-module,
%using an argument similar to the proof of Lemma \ref{lemma:roots-of-t-i}.
Let $\pi:\RR\to\RR/\ZZ=S^1$ be the quotient map. Replacing each $\gamma_i$ by
some power $\gamma_i^{a_i}$ with $a_i$ not divisible by $p_i$, we may assume that
$\pi_j(\gamma_i)=\pi(s_i^{-1})$.
Let $\wt{\gamma}_i\in\Gamma_i'$ satisfy $\eta_i(\wt{\gamma}_i)=\gamma_i$.
Let $o_i$ be the order of $\wt{\gamma}_i$. Then $s_i$ divides $o_i$ and $o_i$ divides $r_i$,
so we may write $o_i=\delta_is_i$ with $1\leq \delta_i\leq CC'$. Passing to a subsequence we may
assume that $o_i=\delta s_i$ for some $1\leq \delta\leq CC'$.
By Lemma \ref{lemma:lifting-to-X-psi} there is a homeomorphism $\theta_i:X_{\psi_i}\to X_{\psi_i}$
lifting the action of $\wt{\gamma}_i$ and such that $\theta_i^{o_i}$ coincides with the
action of $\delta=o_i/s_i$ on $X_{\psi_i}$. Let $w_i=(\chi_i^{-1}\circ\theta_i\chi_i)^*:H_*(X_{\zeta};\ZZ)\to H_*(X_{\zeta};\ZZ)$. Then
$w_i$ is an automorphism of $H_*(X_{\zeta};\ZZ)$ as a $[\ZZ^{\pm t}]$-module, and
$w_i^{o_i}$ is equal to multiplication by $t^\delta$. Since $o_i\to\infty$,
Corollary \ref{cor:finitely-generated-plus}
implies that $H_*(X_{\zeta};\ZZ)$ is a finitely generated $\ZZ$-module.

Let $\kappa$ be the composition
$c_{Y'}\circ \pi_{Y\sharp Y'}:X=(Y\sharp Y')\times Z\to Y\sharp Y'\to Y$
where $\pi_{Y\sharp Y'}$ is the projection to the first factor.
Define $\psi_Y=\pi_j\circ\phi_Y$. Then $\zeta=\psi_Y\circ\kappa$,
so $X_{\zeta}$ can be identified with $(Y_{\psi_Y}\sharp(Y'\times\ZZ))\times Z$,
the product of $Z$ with the connected sum of $Y_{\psi_Y}$
with countably many copies of $Y'$. Since $H_*(Y';\ZZ)\not\simeq H^*(S^n;\ZZ)$,
making connected sum with $Y'$ increases the size of the homology,
so $H_*(X_{\zeta};\ZZ)$ is not finitely generated over $\ZZ$, contradicting our
previous conclusion. This finishes the proof that $\discsym(X)\leq k$.

\section{Proof of Theorem \ref{thm:main-smooth}}
\label{s:proof-thm:main-smooth}

In dimensions up to three any topological manifold has a unique smooth structure, so
we assume that $n\geq 5$.
According to \cite[\S 15A]{Wall}, for any smooth $n$-manifold $X$ and any simple
homotopy equivalence $h:X\to T^n$ one can define a "characteristic class"
$$c(h:X\to T^n)\in A_n:=H^3(T^n;\ZZ/2)\oplus \bigoplus_{i\leq n}H^i(T^n;\pi_i(PL/O))$$
with the property that if $h':X'\to T^n$ is another simple homotopy equivalence
and $c(h:X\to T^n)=c(h':X'\to T^n)$ then $X$ and $X'$ are diffeomorphic.
The piece of the characteristic class in $H^3(T^n;\ZZ/2)$ accounts for the PL
structure of $X$, whereas that in $H^i(T^n;\pi_i(PL/O))$ accounts for the different
choices of smooth structure compatible with the given PL structure.
%The identity map $\Id:T^n\to T^n$ has trivial characteristic class, so
If $c(h:X\to T^n)=0$ then $X$ is diffeomorphic to the standard torus, and
if $\pi:T^n\to T^n$ is a covering and $\pi^*h:\pi^*X\to T^n$ is the pullback of $h$
then
$$c(\pi^*h:\pi^*X\to T^n)=\pi^*c(h:X\to T^n)$$
(note that $\pi^*h:\pi^*X\to T^n$ is also a simple homotopy equivalence).
The homotopy groups $\pi_i(PL/O)$ are finite and hence so is the group $A_n$.

Let $X$ be a smooth $n$-manifold and suppose that $h:X\to T^n$ is a homeomorphism.
Then $h$ is a simple homotopy equivalence by Chapman's theorem (see e.g. the Appendix
in \cite{Cohen}), so we have a
characteristic class $c(h:X\to T^n)\in A_n$. Let $k$ be any natural number and let
$r=k|A_n|+1$. Multiplication by $r$ is the identity map on $A_n$, so if
$\pi_r:T^n\to T^n$ is the covering space defined by $\pi_r(x_1,\dots,x_n)=(rx_1,\dots,rx_n)$
(where $x_i\in\RR/\ZZ$) then $\pi_r^*c(h:X\to T^n)=c(\pi_r^*h:\pi_r^*X\to T^n)=c(h:X\to T^n)$.
Hence there exists a diffeomorphism
$\phi_r:X\to\pi_r^*X$. The manifold $\pi_r^*X$ has a free and smooth
action of $(\ZZ/r)^n$ given by deck transformations of the covering $\pi_r^*X\to X$. This action can be transported
via $\phi_r$ to a free action of $(\ZZ/r)^n$ on $X$. This proves statement (1) of Theorem \ref{thm:main-smooth}.

Let us now prove (2). Let $X$ be a smooth manifold homeomorphic to $T^n$, and fix a homotopy equivalence $h:X\to T^n$.
By Theorems \ref{thm:Minkowski} and \ref{thm:linearising-action}
there exists a natural number $C$ such that for any action of a
finite group $\Gamma$ on
$X$ there is a subgroup $\Gamma_0\leq\Gamma$ satisfying $[\Gamma:\Gamma_0]\leq C$, a map $\psi:X\to T^n$
homotopic to $h$, and a monomorphism $\eta:\Gamma_0\to T^n$ such that $\psi$ is
$\eta$-equivariant.
In particular, the action of $\Gamma_0$ on $X$ is free.

%We next define $\delta(n)$.
%By Lemma \ref{lemma:subgroup-isomorphic-to-Gamma-a-b} there exists a constant $C'$ such that,
%if $\Gamma'$ is any subgroup of $\Gamma_{r,n}$ satisfying $[\Gamma_{r,n}:\Gamma']\leq C$,
%then $\Gamma'$ contains a subgroup isomorphic to $\Gamma_{s,n}$, where $C's\geq r$.
Suppose that $|A_n|=p_1^{e_1}\dots p_k^{e_k}$, where $p_1,\dots,p_k$ are pairwise distinct prime
numbers and each $e_i$ is a natural number. Let $f_i$ be the smallest natural number such that
$p_i^{f_i}\geq C!$. %for every $i$, and define
Define $\delta(n)=p_1^{e_1+f_1}\dots p_k^{e_k+f_k}.$

Let $r$ be an integer multiple of $\delta(n)$ and suppose that $\Gamma'$ is a subgroup of
$\Gamma_{r,n}$ satisfying $[\Gamma_{r,n}:\Gamma']\leq C$.
By Lemma \ref{lemma:subgroup-isomorphic-to-Gamma-a-b} there exists a subgroup $\Gamma''\leq\Gamma'$
isomorphic to $\Gamma_{s,n}$ for some $s$ dividing $r$ and satisfying $C!s\geq r$.
%We next prove that $s$ is divisible by $p_i^{e_i}$ for every $i$.
Let $p_i^{g_i}$ (resp. $p_i^{h_i}$) be the biggest power of $p_i$ dividing $s$ (resp. $r$).
We have $h_i\geq e_i+f_i$ because $r$ is divisible by $\delta(n)$,
and $h_i\geq g_i$ because $s$ divides $r$. Since $r/s\leq C!\leq p_i^{f_i}$
and $p_i^{h_i-g_i}$ divides $r/s$,
we have $h_i-g_i\leq f_i$, so $g_i\geq h_i-f_i\geq e_i$. Hence $s$ is divisible by $|A_n|$.

If the group $\Gamma_{r,n}$ acts smoothly and effectively on $X$ then there is a monomorphism $\eta:\Gamma_{s,n}\to T^n$
and an $\eta$-equivariant map $\psi:X\to T^n$ homotopic to $h$.
The quotient $T^n\to T^n/\eta(\Gamma_{s,n})$ is a covering map and $T^n/\eta(\Gamma_{s,n})$
is homeomorphic to $T^n$. So the map $\psi$ descends to a continuous map $\zeta:X/\Gamma_{s,n}\to T^n/\eta(\Gamma_{s,n})\cong T^n$.
The projection map identifies $\pi_1(X)$ (resp. $\pi_1(T^n)$) with a subgroup of $\pi_1(X/\Gamma_{s,n})$
(resp. $\pi_1(T^n/\eta(\Gamma_{s,n})$), and via these identifications
$\zeta_*:\pi_1(X/\Gamma_{s,n})\to \pi_1(T^n/\eta(\Gamma_{s,n}))$ extends
$\psi_*:\pi_1(X)\to\pi_1(T^n)$. Since $\psi_*$ is an isomorphism,
$[\pi_1(X/\Gamma_{s,n}):\pi_1(X)]=|\Gamma_{s,n}|=[\pi_1(T^n/\eta(\Gamma_{s,n})):\pi_1(T^n)]$,
and $\pi_1(X/\Gamma_{s,n})\simeq \pi_1(T^n/\eta(\Gamma_{s,n}))\simeq\ZZ^n$,
it follows that $\zeta_*$ is an isomorphism. Both $X/\Gamma_{s,n}$ and $T^n/\eta(\Gamma_{s,n})$ are aspherical
spaces, so $\zeta$ is a homotopy equivalence.
By the topological rigidity
of tori, $\zeta$ is homotopic to a homeomorphism $\xi:X/\Gamma_{s,n}\to T^n$. Since $\xi$ is homotopic to
$\zeta$, it can be lifted to a homeomorphism $\theta:X\to T^n$ that makes the following diagram commutative:
$$\xymatrix{X\ar[r]^{\theta}\ar[d] & T^n \ar[d]^q \\ X/\Gamma_{s,n}\ar[r]^-{\xi} & T^n.}$$
The map $q$ is given by $q(x_1,\dots,x_n)=(sx_1,\dots,sx_n)$,
so its action on $H^k(T^n;\ZZ)$ is multiplication by $s^k$.
Since $s$ is divisible by $|A_n|$, the universal coefficients theorem implies that
$c(\theta:X\to T^n)=q^* c(\xi:X/\Gamma_{s,n}\to T^n)=0$, so $X$ is diffeomorphic to $T^n$.

\section{Holomorphic finite group actions on Kaehler manifolds}
\label{s:holomorphic}

\subsection{Proof of Theorem \ref{thm:Lie-groups}}
The implication (1)$\Rightarrow$(2) is immediate, as the $r$-torsion of $T^n$
is isomorphic to $(\ZZ/r)^n$. To prove the converse implication
(2)$\Rightarrow$(1) it is enough to consider the case in which
$G$ is compact. Indeed, if $G$ is an arbitrary Lie group with finitely many connected components
then the existence and uniqueness up to conjugation of maximal compact subgroups
(see e.g. \cite[Theorem 14.13]{HN}) implies the existence of a compact subgroup $K\leq G$ with the property
that any compact (in particular, any finite) subgroup of $G$ is conjugate to a subgroup
of $K$. Hence, replacing $G$ by $K$ we assume from now on that $G$ is a compact Lie group.

The proof can be finished with elementary arguments using the exponential map
and the adjoint representation. Instead, we give a short but overkill argument. Let $T\leq G$
be a maximal torus. Then $G/T$ has a natural structure of smooth manifold and $\chi(G/T)$ is
nonzero (see e.g. \cite[Prop. 17.4]{Bump}). By \cite[Theorem 2.5]{Mundet2022} there exists some
constant $C$ such that any finite group $\Gamma$ acting continuously on $G/T$ has a subgroup
$\Gamma'\leq\Gamma$ satisfying $[\Gamma:\Gamma']\leq C$ and fixing some point in $G/T$.
Now suppose that, for some integer $a\geq C!+1$, $G$ has a subgroup $\Gamma$ isomorphic to
$\Gamma_{a,n}$. Consider the action of $\Gamma$ on the left on $G/T$. There is a subgroup
$\Gamma'\leq\Gamma$ fixing a point $gT\in G/T$, so $g\Gamma'g^{-1}$
is contained in $T$. By Lemma \ref{lemma:subgroup-isomorphic-to-Gamma-a-b}
there is a subgroup of $\Gamma'$ isomorphic to $\Gamma_{a',n}$ for some $a'$ satisfying
$C!a'\geq a$, hence $a'\geq 2$. By Lemma \ref{lemma:upper-bound-subgroup-T-d} it follows that
$\dim T\geq n$.

\subsection{Proof of Theorem \ref{thm:Kaehler}}
\label{ss:proof-thm:Kaehler}
Let $X$ be a compact connected Kaehler manifold of real dimension $n$. Let $\Aut X$ denote the group of
biholomorphisms of $X$. A theorem of Bochner and Montgomery \cite[\S 9]{BM} states that $\Aut X$ has
a natural structure of Lie group. Let $\omega$ denote the Kaehler form of $X$ and let
$\ov{\omega}$ denote the cohomology class in $H^2(X;\RR)$ represented by $\omega$ through the
de Rham isomorphism. Let $\Aut_{\ov{\omega}}X$ denote the subgroup of $\Aut X$ consisting of
those biholomorphisms fixing the class $\ov{\omega}$. According to a theorem of Fujiki \cite[Theorem 4.8]{Fuj},
$\Aut_{\ov{\omega}}X$ has finitely many connected components.

By Theorem \ref{thm:Minkowski} there exists a constant $C$ depending only on $X$
such that for any action of a finite group  $\Gamma$ on $X$ then there is a subgroup
$\Gamma'\leq\Gamma$ satisfying $[\Gamma:\Gamma']\leq C$ and whose action on $H^2(X;\ZZ)$,
and hence on $H^2(X;\RR)$, is trivial.
In particular, if $\Gamma$
acts effectively on $X$ by biholomorhic transformations, so that we can identify $\Gamma$ with a subgroup
of $\Aut X$, then $[\Gamma:\Aut_{\ov{\omega}}X\cap\Gamma]\leq C$.

Now assume that for some natural number $m$ the group $\Aut X$ contains subgroups isomorphic
to $\Gamma_{r,m}$ for arbitrary large values of $r$. Arguing as in the preceding subsection
(applying the existence of the constant $C$ and using Lemma \ref{lemma:subgroup-isomorphic-to-Gamma-a-b})
we may conclude that $\Aut_{\ov{\omega}}X$ contains subgroups isomorphic to $\Gamma_{s,m}$ for
arbitrary large values of $s$. This implies, by Theorem \ref{thm:Lie-groups}, that
$\Aut_{\ov{\omega}}X$ contains a torus $T$ of satisfying $\dim T\geq m$.

By the principal orbit theorem (see e.g. \cite[Theorem (5.14)]{tD}),
if an $m$-dimensional torus $T$ acts effectively on an $n$-dimensional
connected topological manifold $X$ then $m\leq n$, and if $m=n$ then $X$
is homeomorphic to a torus.
%
%Since we did not find a proof in the literature, we sketch it here.
%For any nontrivial closed subgroup $K\leq T$ the fixed point set $X^K$ is closed and has empty %interior
%(see \cite[Chap III, Theorem 9.5]{Bredon}). The set of closed subgroups of $T$ is countable
%and topological manifolds are Baire spaces, so $\bigcap_{\{e\}\neq K\leq T}X\setminus %X^K\neq\emptyset$.
%Hence, there exists a point $x\in X$ whose stabilizer is trivial, so
%the map $f:T\ni t\mapsto t\cdot x\to X$ is continuous and injective. By the theorem of invariance %of domain
%it follows that $\dim T\leq\dim X$ and that, if $\dim T=\dim X$, then $f$ is an open map. In that
%case, $f(T)$ is closed (because $T$ is compact) and open in $X$ and since $X$ is connected it %follows that $f(T)=X$. Since
%$f$ is open and injective, it gives a homeomorphism $T\cong f(T)$, so we conclude that $T\cong X$.
%
A Kaehler manifold homeomorphic to a torus is biholomorphic to a complex torus (see e.g.
\cite[Theorem B]{BC} for a nice exposition of a more general result), so the proof of Theorem \ref{thm:Kaehler}
is now finished.

\section{Regular self coverings of the manifolds in Theorem \ref{thm:CWY}}
\label{s:CWY}

Let $d$ be an odd natural number.
The manifolds constructed in \cite{CWY} are products $T(h)\times H$, where $T(h)$ is the mapping
torus of a self homeomorphism $h$ of a closed topological manifold $V$ and $H$ is a closed hyperbolic manifold.
So it suffices to prove that $T(h)$ supports a regular self covering of degree $d$. The structure of
mapping torus on $T(h)$ gives a map $T(h)\to S^1$, and the regular self covering we claim to exist
is the pullback, via this map,
of the covering $S^1\to S^1$ sending $\theta$ to $d\,\theta$. This pullback can be identified with the
mapping torus $T(h^d)$, so all we need to prove is that $T(h)$ and $T(h^d)$ are homeomorphic.

The manifold $V$ is $W\cup_T (T^n\times[0,1])\cup_{T'}W'$, where $W$ and $W'$ are $(n+1)$-dimensional manifolds with boundaries
$T$ and $T'$ respectively, and where $n\geq 5$.
Here both $T$ and $T'$ denote the torus $T^n$ with involutions $h_T:T\to T$ and
$h_{T'}:T'\to T'$. The involution $h_T$ is a linear involution, whereas $h_{T'}$ is exotic, i.e., not conjugate to a linear action. Both $h_T$ and $h_{T'}$ have nonempty
fixed point set (see \cite[\S 2.1]{BW} for a concrete description of the involutions $h_T$ and $h_{T'}$ used in
\cite{CWY}). The maps $h_T$ and $h_{T'}$ are homotopic.
In the definition of $V$ we glue $T\subset W$ with $T^n\times\{0\}$ and $T'\subset W'$ with
$T^n\times\{1\}$.

The involutions $h_T$ and $h_{T'}$ extend to involutions $h_W:W\to W$
and $h_{W'}:W'\to W'$ respectively, and there is a self homeomorphism $h_C$ of $C:=T^n\times [0,1]$
whose restriction to $C_0:=T^n\times\{0\}$ resp. $C_1:=T^n\times\{1\}$ coincides with $h_T$ resp. $h_{T'}$.
To justify the existence of $h_C$ it suffices to prove the existence of a homeomorphism
$\phi:C\to C$ such that $\phi|_{C_0}=\Id_{T^n}$ and $\phi|_{C_1}=\psi:=h_T^{-1}\circ h_{T'}$, for then
$h_C:=(h_T\times\Id_{[0,1]})\circ \phi$ has the desired property. Now, $\psi$ is homotopic to $\Id_{T^n}$
so the mapping tori $T(\psi)$ and $T(\Id_{T^n})$ are homotopy equivalent. Since $T(\Id_{T^n})=T^{n+1}$,
the topological rigidity of tori implies that $T(\psi)$ and $T(\Id_{T^n})$ are homeomorphic.
The existence of $\phi$ now follows from combining: \cite[Theorem 1]{Lawson}, the observation that invertible
cobordism are $h$-cobordisms, the $s$-cobordism theorem, and the vanishing of the Whitehead group of
$\pi_1(T^n)$.

Unlike $h_T$ and $h_{T'}$, $h_C$ is not an involution. However, we have the following:

\begin{prop}
\label{prop:homotopy-rel-bdry}
If we chose $h_C$ suitably, then
$h_C^2$ and $\Id_C$ are homotopic rel. $\partial C$.
\end{prop}
\begin{proof}
We identify $T^n$ with $\RR^n/\ZZ^n$, so the universal covering space $C^{\sharp}$ of $C$ can be
identified with $\RR^n\times [0,1]$. Let $C^{\sharp}_i=\RR^n\times\{i\}$ for $i=0,1$.
Let $f,g:C\to C$ be continuous maps such that
$f|_{\partial C}=g|_{\partial C}$. Choose lifts $f^{\sharp},g^{\sharp}:C^{\sharp}\to C^{\sharp}$
of $f,g$ respectively. Then
$g^{\sharp}|_{C^{\sharp}_i}-f^{\sharp}|_{C^{\sharp}_i}$ is equal to some constant $\lambda_i\in\ZZ^n$, because
$g|_{C_i}=f|_{C_i}$. Let
$$\lambda(g,f):=\lambda_1-\lambda_0.$$
The vector $\lambda(g,f)\in\ZZ^n$ is independent of the chosen lifts of $f,g$.

\begin{lemma}
$f$ and $g$ are homotopic rel. $\partial C$ if and only if $\lambda(g,f)=0$.
\end{lemma}
\begin{proof}
The "only if" part of the lemma is an easy exercise. For the "if" part, note that
there is a linear map $\rho:\ZZ^n\to\ZZ^n$ such that
for every $p\in\RR^n$, $\mu\in\ZZ^n$ and $s\in[0,1]$,
both $f^{\sharp}(p+\mu,s)-f^{\sharp}(p,s)$ and
$g^{\sharp}(p+\mu,s)-g^{\sharp}(p,s)$ are equal to $(\rho(\mu),0)$.
%
%$f^{\sharp}(p+\mu,s)=f^{\sharp}(p,s)+(\rho(\mu),0)$
%and $g^{\sharp}(p+\mu,s)=g^{\sharp}(p,s)+(\rho(\mu),0)$ for every %$p\in\RR^n,\,\mu\in\ZZ^n,\,s\in[0,1]$.
%
(Actually $\rho$ can be identified with the morphism $H_1(T^n)\to H_1(T^n)$ induced
by $f$ or $g$.) It follows that the map $C^{\sharp}\times [0,1]\to C^{\sharp}$
sending $((p,s),t)$ to $(1-t)f^{\sharp}(p,s)+tg^{\sharp}(p,s)$,
which is a homotopy between $f^{\sharp}$ and $g^{\sharp}$,
descends to a homotopy rel $\partial C$ between $f$ and $g$.
\end{proof}

Now suppose that $\zeta:C\to C$ is a homeomorphism
satisfying $\zeta|_{C_0}=h_T$ and $\zeta|_{C_1}=h_{T'}$.
Since both $h_T$ and $h_{T'}$ have fixed points, there exist $x\in C_0$ and $y\in C_1$
such that $\zeta(x)=x$ and $\zeta(y)=y$. Choose lifts $x^{\sharp},y^{\sharp}\in C^{\sharp}$ of $x,y$ respectively.
There is a unique lift $\zeta^{\sharp}:C^{\sharp}\to C^{\sharp}$ satisfying $\zeta^{\sharp}(x^{\sharp})=x^{\sharp}$.
As before, there is a morphism of groups $\rho:\ZZ^n\to\ZZ^n$ such that
$\zeta^{\sharp}(p+\mu,s)=\zeta^{\sharp}(p,s)+(\rho(\mu),0)$ for all $p,\mu,s$.
Let $o:C^{\sharp}=\RR^n\times [0,1]\to\RR^n$ denote the projection and let
$\nu:=o(\zeta^{\sharp}(y^{\sharp})-y^{\sharp})\in\ZZ^n$, so that $\zeta^{\sharp}(y^{\sharp})=y^{\sharp}+(\nu,0)$. Then:
\begin{align*}
\lambda(\zeta^2,\Id) &= o(\zeta^{\sharp}\zeta^{\sharp}(y^{\sharp})-y^{\sharp})
= o(\zeta^{\sharp}\zeta^{\sharp}(y^{\sharp})-\zeta^{\sharp}(y^{\sharp})+\zeta^{\sharp}(y^{\sharp})-y^{\sharp}) \\
&=o(\zeta^{\sharp}(y^{\sharp}+(\nu,0))-\zeta^{\sharp}(y^{\sharp})+(\nu,0))
=\rho(\nu)+\nu.
\end{align*}
By the previous lemma, in order for $\zeta^2$ to be homotopic to $\Id_C$ rel. $\partial C$
we need $\rho(\nu)+\nu$ to vanish. This need not be the case, but if we define
$\xi^{\sharp}:C^{\sharp}\to C^{\sharp}$
as $\xi^{\sharp}(p,s)=\zeta^{\sharp}(p)-s\nu$ then we have
$\xi^{\sharp}(p+\mu,s)=\xi^{\sharp}(p,s)+(\rho(\mu),0)$ for all $p,\mu,s$,
so $\xi^{\sharp}$ descends to a homeomorphism $\xi:C\to C$ satisfying $\xi|_{\partial C}=\zeta|_{\partial C}$.
Furthermore, $\xi^{\sharp}(x^{\sharp})=x^{\sharp}$ and $\xi^{\sharp}(y^{\sharp})=y^{\sharp}$, so
$\lambda(\xi^2,\Id)=0$. Consequently, $h_C:=\xi$ has the desired property.
\end{proof}

Assume from now on that $h_C$ has been chosen in such a way that $h_C^2$ is homotopic to $\Id_C$
rel. $\partial C$, which implies that $h_C^d$ and $h_C$ are homotopic rel. $\partial C$.
The involution $h:V\to V$ is defined by the condition that its restriction to the subspaces
$W,\,T^n\times [0,1],\,W'$ is given by $h_W,\,h_C,\,h_{W'}$ respectively.

Since $h_W$ and $h_{W'}$ are involutions, the restrictions of the maps $h$ and $h^d$ to $W$ and $W'$ are equal.
Hence, to prove that $T(h)$ and $T(h^d)$ are homeomorphic
it suffices to prove the existence of a homeomorphism of mapping tori $\phi:T(h_C)\to T(h_C^d)$ whose restriction
to $\partial T(h_C)$ is the natural homeomorphism $\partial T(h_C)\to\partial T(h_C^d)$ resulting from
the equalities $h_T=h_T^d$ and $h_{T'}=h_{T'}^d$.
%We are going to prove the existence of $\phi$ using
%the topological rigidity of non-positively curved Riemannian manifolds proved in %\cite{Farrell,FJ}.
%As in \cite{CWY} we implicitly assume that $n\geq 5$, which is required for the results in \cite{Farrell,FJ} to %apply.

By definition $T(h_C)$ is the quotient of $C\times [0,1]$ by the relation that
identifies $(x,1)$ with $(h_C(x),0)$ for every $x\in C$.
For $i=0,1$, let $T_i(h_C)\subset T(h_C)$ be the image of
$C_i\times [0,1]$ under the projection map $C\times [0,1]\to T(h_C)$.
Then $T_0(h_C)$ resp. $T_1(h_C)$ can be identified with $T(h_T)$ resp. $T(h_{T'})$.
Since $h_T$ is a linear involution, $T(h_T)$ supports a non-positively curved Riemannian metric.
Now, $T(h_{T'})$ is homotopy equivalent to $T(h_T)$, because $h_T$ and $h_{T'}$ are homotopic.
Hence, by topological rigidity \cite[Theorem 14.1]{Farrell}, there is a homeomorphism
$\psi:T(h_{T'})\to T(h_T)$. Choosing $\psi$ appropriately, we may and do assume that the compositions
of maps
$T(h_T)\stackrel{\psi^{-1}}{\longrightarrow}T(h_{T'})=T_1(h_C)\hookrightarrow T(h_C)$ and
$T(h_T)=T_0(h_C)\hookrightarrow T(h_C)$ are homotopic.

We claim the existence of a homeomorphism
$$\xi:T(h_C)\to T(h_T\times\Id_{[0,1]})=T(h_T)\times [0,1]$$
whose restriction
to $T_0(h_C)$ resp. $T_1(h_C)$ coincides with $\Id_{T(h_T)}$ resp. $\psi$. Once we prove the claim
the existence of $\phi$ will follow immediately, since by \cite[Theorem 14.1]{Farrell} topological
rigidity applies to $T(h_T)\times [0,1]$ (again because $T(h_T)$ supports a metric
of non-positive curvature) and the existence of a homotopy $h_C^d\sim h_C$ rel. $\partial C$
gives a homotopy equivalence $T(h_C)\to T(h_C^d)$ whose restriction to $\partial T(h_C)$ is a
homeomorphism.

To prove the existence of the homeomorphism $\xi:T(h_C)\to T(h_T)\times [0,1]$
we rely once again on the topological rigidity of $T(h_T)\times [0,1]$, so we only
need to prove the existence of a continuous map $\chi:T(h_C)\to T(h_T)\times [0,1]$
whose restriction to $T_0(h_C)$ resp. $T_1(h_C)$ coincides with $\Id_{T(h_T)}$ resp. $\psi$
(these properties imply that $\chi$ is a homotopy equivalence).

If $Y\subseteq X$ is an inclusion of
topological spaces and $f,g:X\to X$ are maps preserving $Y$, satisfying $f|_Y=g|_Y$,
and $f,g$ are homotopic rel. $Y$, then there is a continuous map
$\epsilon:T(f)\to T(g)$ whose restriction to $T(f|_Y)$ is the natural identification
between $T(f|_Y)$ and $T(g|_Y)$. Indeed, suppose that
$H:X\times I\to X$ satisfies $H(x,0)=g(x)$,
$H(x,1)=f(x)$ and $H(y,t)=f(y)=g(y)$ for every $x\in X$, $y\in Y$, $t\in [0,1]$. Then
$\epsilon$ is defined by the map $\wt{\epsilon}:X\times [0,1]\to X\times [0,1]$ given by
$$\wt{\epsilon}(x,t)=\left\{\begin{array}{ll}
(x,2t) & \qquad \text{if $t\in[0,1/2],$} \\
(H(x,2t-1),0) & \qquad \text{if $t\in [1/2,1]$.}\end{array}\right.$$

Using the previous
principle, and the facts that $h_C$ and $h_T\times \Id_{[0,1]}$ are homotopic rel. $C_0$ and
$h_C$ and $h_{T'}\times \Id_{[0,1]}$ are homotopic rel. $C_1$
(which can be proved by lifting the maps to $C^{\sharp}$ as in the proof of Proposition
\ref{prop:homotopy-rel-bdry} and interpolating linearly),
we deduce the existence of maps $\chi_0:T(h_C)\to T(h_T)\times [0,1]$
and $\chi_1:T(h_C)\to T(h_{T'})\times [0,1]$
such that $\chi_i$ restricted to $T_i(h_C)$ is the identity for $i=0,1$.
Furthermore, $\chi_0$ and $(\psi\times\Id_{[0,1]})\circ\chi_1$ are homotopic.

The universal cover of $T(h_C)$
can be identified with $\RR^n\times [0,1]\times \RR$, and that of
$T(h_T)\times [0,1]$ with $\RR^n\times\RR\times [0,1]$.
Fix a lift $h_T^{\sharp}:\RR^n\to\RR^n$ of $h_T:T^n\to T^n$.
Crucially, $h_T^{\sharp}$ is an affine isomorphism, because $h_T$ is a linear involution.
The group $\Gamma=\pi_1(T(h_T)\times [0,1])\simeq\pi_1(T(h_T))\simeq \ZZ^n\rtimes \ZZ$
acts on $\RR^n\times\RR\times [0,1]$ preserving the affine structure induced
by the inclusion $\RR^n\times\RR\times [0,1]\subset\RR^{n+2}$: the factor $\ZZ^n$
acts by addition on the first factor of $\RR^n\times\RR\times [0,1]$, and the action
of the second factor is generated by the transformation $(z,t,s)\mapsto (h_T^{\sharp}(z),t-1,s)$.

Choose lifts of $\chi_0$ and $(\psi\times\Id_{[0,1]})\circ\chi_1$
to the universal coverings, and call them $\theta_0$ and $\theta_1$ respectively,
so that
$\theta_i:\RR^n\times [0,1]\times \RR\to \RR^n\times \RR\times [0,1].$
Since $\theta_0$ and $\theta_1$ are homotopic there exists a morphism of groups
$\rho:\pi_1(T(h_C))\to \pi_1(T(h_T)\times [0,1])$ such that both $\theta_0$
and $\theta_1$ are $\rho$-equivariant, meaning that
$\theta_i(\gamma\cdot w)=\rho(\gamma)\cdot \theta_i(w)$ for every
$\gamma\in\pi_1(T(h_C))$ and $w\in\RR^n\times [0,1]\times \RR$.
Define
$\theta:\RR^n\times [0,1]\times \RR\to \RR^n\times \RR\times [0,1]$
as
$\theta(p,s,t)=(1-s)\theta_0(p,s,t)+s\theta_1(p,s,t)$.
Then $\theta$ satisfies the same $\rho$-equivariance property
as $\theta_i$, because the action of $\pi_1(T(h_T)\times [0,1])$ on $\RR^n\times \RR\times [0,1]$
is affine. This implies that $\theta$ descends to a map $\chi:T(h_C)\to T(h_T)\times [0,1]$,
and by construction the restriction of $\chi$ to $T_0(h_C)$ resp. $T_1(h_C)$ coincides
with $\Id_{T(h_T)}$ resp. $\psi$. This finishes the proof of the theorem.

%\vspace*{1cm}
%
%\noindent{\large \bf Declarations}
%
%\vspace*{0,25cm}
%
%\noindent{\bf Conflict of interest.} The corresponding author states that there is no conflict of
%interest.
%
%\vspace*{0,15cm}
%
%\noindent{\bf Data availability.} Data sharing not applicable to this article as no datasets were generated or analysed during
%the current study.
%
%\vspace*{0,15cm}
%
%\noindent{\bf Code availability.} Not applicable.


\begin{thebibliography}{99}

\bibitem{AP} C. Allday, V. Puppe,
{\it Cohomological methods in transformation groups},
Cambridge Studies in Advanced Mathematics, 32. Cambridge University Press, Cambridge, 1993.

\bibitem{AssadiBurghelea}
A. Assadi, D. Burghelea,
Examples of asymmetric differentiable manifolds.
{\em Math. Ann.} {\bf 255} (1981), no. 3, 423--430.

\bibitem{AM}
M.F. Atiyah, I.G. Macdonald, Introduction to commutative algebra,
Addison--Wesley, 1969.


\bibitem{BC}
O. Baues, V. Cort\'es,
Aspherical K\"ahler manifolds with solvable fundamental group,
{\em Geom. Dedicata} {\bf 122} (2006), 215--229.

\bibitem{BKKPR}
S. Behrens, B. Kalmar, M. H. Kim, M. Powell, A. Ray (editors), The Disc Embedding Theorem,
Oxford Univ. Press, Oxford, 2021.

\bibitem{Berrick}
A.J. Berrick, A topologist's view of perfect and acyclic groups.
Invitations to geometry and topology, 1--28, Oxf. Grad. Texts Math., 7, Oxford Univ. Press, Oxford, 2002.


\bibitem{BBMBP}
L. Bessières, G. Besson, S. Maillot, M. Boileau, J. Porti,
Geometrisation of 3-manifolds.
EMS Tracts in Mathematics, 13. European Mathematical Society (EMS), Zürich, 2010.

\bibitem{BW}
J. Block, S. Weinberger,
On the generalized Nielsen realization problem,
{\em Comment. Math. Helv.} {\bf 83} (2008), no. 1, 21--33.

\bibitem{BM}
S. Bochner, D. Montgomery,
Locally compact groups of differentiable transformations,
{\em Ann. of Math.} (2) {\bf 47} (1946), 639--653.

\bibitem{Bo} A. Borel,
{\em Seminar on transformation groups}, Ann. of Math. Studies {\bf 46},
Princeton University Press, N.J., 1960.

\bibitem{Bredon}
G.E. Bredon,
{\em Introduction to compact transformation groups},
Pure and Applied Mathematics, Vol. {\bf 46}, Academic Press, New York-London (1972).

\bibitem{BtD}
T. Br\"ocker, T. tom Dieck,
Representations of compact Lie groups. Translated from the German manuscript.
Corrected reprint of the 1985 translation. Graduate Texts in Mathematics {\bf 98}.
Springer-Verlag, New York, 1995. x+313 pp.

\bibitem{Bump}
D. Bump, Lie groups. Graduate Texts in Mathematics, 225. Springer-Verlag, New York, 2004.

\bibitem{BS}
D. Burghelea, R. Schultz,
On the semisimple degree of symmetry,
{\em Bull. Soc. Math. France} {\bf 103} (1975), no. 4, 433--440.

\bibitem{BKKT}
M. Bustamante, M. Krannich, A. Kupers, B. Tshishiku,
Mapping class groups of exotic tori and actions by $\SL_d(\ZZ)$,
{\it preprint} {\tt arXiv:2305.08065}.


\bibitem{BT}
M. Bustamante, B. Tshishiku, Symmetries of exotic smoothings of aspherical space forms, {\it preprint} {\tt arXiv:2109.0919}.

\bibitem{CWY}
S. Cappell, S. Weinberger, M. Yan,
Closed aspherical manifolds with center,
{\em J. Topol.} {\bf 6} (2013), no. 4, 1009--1018.

\bibitem{Carlsson}
G. Carlsson,
On the homology of finite free (Z/2)n-complexes,
{\em Invent. Math.} {\bf 74} (1983), no. 1, 139--147.

\bibitem{CW}
S. Chang, S. Weinberger, A course on surgery theory. Annals of Mathematics Studies, 211. Princeton University Press, Princeton, NJ, 2021.

\bibitem{Cohen}
M.M. Cohen,
A course in simple-homotopy theory.
Graduate Texts in Mathematics, Vol. 10. Springer-Verlag, New York-Berlin, 1973.

%\bibitem{CRW}
%P.E. Conner, F. Raymond, P.J. Weinberger,
%Manifolds with no periodic maps. Proceedings of the Second Conference on Compact Transformation %Groups (Univ. Massachusetts, Amherst, Mass., 1971), Part II, pp. 81--108. Lecture Notes in Math., %Vol. 299, Springer, Berlin, 1972.

\bibitem{CMPS}
B. Csik\'os, I. Mundet i Riera, L. Pyber, E. Szab\'o, Number of stabilizer subgroups in a finite
group acting on a manifold, {\it preprint} {\tt arXiv:2111.14450}.

%\bibitem{D}
%G. D'Ambra, Isometry groups of Lorentz manifolds, {\em Invent. Math.} {\bf 92} (1988), 555--565.

\bibitem{DS} H. Donnelly, R. Schultz,
Compact group actions and maps into aspherical manifolds,
{\em Topology} {\bf 21} (1982), no. 4, 443--455.

\bibitem{Edmonds}
A.L. Edmonds,
Construction of group actions on four-manifolds,
{\em Trans. Amer. Math. Soc.} {\bf 299} (1987), no. 1, 155--170.

\bibitem{Farrell}
F.T. Farrell,
Lectures on surgical methods in rigidity.
Published for the Tata Institute of Fundamental Research, Bombay; by Springer-Verlag, Berlin, 1996.

\bibitem{FJ}
F.T. Farrell, L.E. Jones,
Topological rigidity for compact non-positively curved manifolds.
Differential geometry: Riemannian geometry (Los Angeles, CA, 1990), 229--274,
Proc. Sympos. Pure Math., 54, Part 3, Amer. Math. Soc., Providence, RI, 1993.

\bibitem{FQ}
M.H.  Freedman, F. Quinn, Topology of 4-manifolds. Princeton Mathematical Series, 39.
Princeton University Press, Princeton, NJ, 1990.

\bibitem{FrTa}
A. Fr\"ohlich, M.J. Taylor,
Algebraic number theory. Cambridge Studies in Advanced Mathematics, 27.
Cambridge University Press, Cambridge, 1993.



\bibitem{Fuj}
A. Fujiki, On automorphism groups of compact Kaehler manifolds, {\em Invent. Math.} {\bf 44}
(1978) 225--258.

\bibitem{Gol}
A. Golota, Finite abelian subgroups in the groups of birational and bimeromorphic selfmaps,
{\it preprint} {\tt  arXiv:2205.00607}.

\bibitem{GLO}
D.H. Gottlieb, K.B. Lee, M. \"Ozaydin, Compact group actions and maps into $K(\pi,1)$-spaces.
{\em Trans. Amer. Math. Soc.} {\bf 287} (1985), no. 1, 419--429.

\bibitem{Hanke}
B. Hanke,
The stable free rank of symmetry of products of spheres,
{\em Invent. Math.} {\bf 178} (2009), no. 2, 265--298.
Erratum to: The stable free rank of symmetry of products of spheres,
{\em Invent. Math.} {\bf 182} (2010), no. 1, 229.

\bibitem{Hatcher}
A. Hatcher, Algebraic topology. Cambridge University Press, Cambridge, 2002.

\bibitem{HN}
J. Hilgert, K.-H. Neeb,
Structure and geometry of Lie groups.
Springer Monographs in Mathematics. Springer, New York, 2012.

\bibitem{Hsiang}
W.-y. Hsiang,
Cohomology theory of topological transformation groups.
Ergebnisse der Mathematik und ihrer Grenzgebiete, Band 85. Springer-Verlag, New York-Heidelberg, 1975.

\bibitem{Hsiang1969}
W.-y. Hsiang,
On the bound of the dimensions of the isometry groups of all possible riemannian metrics on an exotic sphere.
Ann. of Math. (2) 85 (1967), 351–358.

\bibitem{HH}
W.-c. Hsiang, W.-y. Hsiang,
The degree of symmetry of homotopy spheres,
{\em Ann. of Math.} (2) {\bf 89} (1969), 52--67.

\bibitem{HW}
W.-c. Hsiang, C.T.C. Wall,
On homotopy tori. II,
{\em Bull. London Math. Soc.} {\bf 1} (1969), 341--342.



\bibitem{Huck}
W. Huck, Circle actions on 4-manifolds. I,
{\em Manuscripta Math.} {\bf 87} (1995), no. 1, 51--70.

\bibitem{Kervaire}
M.A. Kervaire,
Smooth homology spheres and their fundamental groups,
{\em Trans. Amer. Math. Soc.} {\bf 144} (1969), 67--72.


\bibitem{KL}
M. Kreck, W. L\"uck,
Topological rigidity for non-aspherical manifolds,
{\em Pure Appl. Math. Q.} {\bf 5} (2009), no. 3,
Special Issue: In honor of Friedrich Hirzebruch. Part 2, 873--914.

\bibitem{KK}
H.T. Ku, M.C. Ku, Group actions on $A_k$-manifolds,
{\em Trans. Amer. Math. Soc.} {\bf 245} (1978), 469--492.



\bibitem{KwasikSchultz}
S. Kwasik, R. Schultz,
Isolated fixed points of circle actions on 4-manifolds,
{\em Forum Math.} {\bf 9} (1997), no. 4, 517--546.

\bibitem{Lawson}
T.C. Lawson,
Splitting isomorphisms of mapping tori,
{\em Trans. Amer. Math. Soc.} {\bf 205} (1975), 285--294.

\bibitem{LeeRaymond}
K.B. Lee, F. Raymond,
Topological, affine and isometric actions on flat Riemannian manifolds,
{\em J. Differential Geometry} {\bf 16} (1981), no. 2, 255--269.

\bibitem{LeeRaymond2}
K. B. Lee, F. Raymond,
Seifert fiberings.
Mathematical Surveys and Monographs, 166. American Mathematical Society, Providence, RI, 2010.

\bibitem{L}
W. L\"uck,
Some open problems about aspherical closed manifolds.
Trends in contemporary mathematics, 33--46,
Springer INdAM Ser., 8, Springer, Cham, 2014.

\bibitem{MS}
L. N. Mann, J. C. Su, Actions of elementary p-groups on manifolds,
{\em Trans. Amer. Math. Soc.} {\bf 106} (1963), 115--126.

\bibitem{Mat}
H. Matsumura, Commutative ring theory. Translated from the Japanese
by M. Reid. Second edition. Cambridge Studies in Advanced Mathematics, 8.
Cambridge University Press, Cambridge, 1989.

\bibitem{Min} H. Minkowski, Zur Theorie der positiven quadratischen Formen, {\em Journal fur die reine und angewandte Mathematik} {\bf 101} (1887), 196-202.

\bibitem{MorganTian}
J. Morgan, G. Tian,
The geometrization conjecture. Clay Mathematics Monographs, 5.
American Mathematical Society, Providence, RI; Clay Mathematics Institute, Cambridge, MA, 2014.

\bibitem{M1} I. Mundet i Riera, Jordan's theorem for the
    diffeomorphism group of some manifolds {\em Proc. Amer.
    Math. Soc.} {\bf 138} (2010), 2253-2262.

\bibitem{Mundet2016} I. Mundet i Riera, Finite group actions on 4-manifolds with nonzero Euler
characteristic, {\em Math. Z.} {\bf 282} (2016), 25--42.

\bibitem{M6} I. Mundet i Riera,
Non Jordan groups of diffeomorphisms and actions of compact Lie groups on manifolds,
{\it Transformation Groups} {\bf 22} (2017), no. 2, 487--501,
DOI 10.1007/s00031-016-9374-9.

\bibitem{Mundet2022}
I. Mundet i Riera,
Jordan property for homeomorphism groups and almost fixed point property,
{\it preprint} {\tt arXiv:2210.07081}.

\bibitem{MundetSaez}
I. Mundet i Riera, C. S\'aez-Calvo,
Which finite groups act smoothly on a given 4-manifold?
{\em Trans. Amer. Math. Soc.} {\bf 375} (2022), no. 2, 1207--1260.

\bibitem{Pardon2019}
J. Pardon, Smoothing finite group actions on three-manifolds,
{\em Duke Math. J.} {\bf 170} (2021), no. 6, 1043--1084.


\bibitem{PS}
Y. Prokhorov, C. Shramov,
Jordan property for Cremona groups,
{\em Amer. J. Math.} {\bf 138} (2016), no. 2, 403--418.

%\bibitem{Puppe}
%V. Puppe,
%Do manifolds have little symmetry?
%{\em J. Fixed Point Theory Appl.} {\bf 2} (2007), no. 1, 85--96.

%\bibitem{Rob} D.J.S. Robinson, {\em A course in the theory of
%    groups}, Second edition, Graduate Texts in Mathematics {\bf
%    80}, Springer-Verlag, New York, 1996.

\bibitem{QSW}
L. Qin, Y. Su, B. Wang,
Self-Covering, finiteness, and fibering over a circle,
{\it preprint} {\tt arXiv:2112.11750v2}.



\bibitem{Schultz-1}
R. Schultz, Group actions on hypertoral manifolds. I. Topology Symposium, Siegen 1979 (Proc. Sympos., Univ. Siegen, Siegen, 1979), pp. 364--377, Lecture Notes in Math., 788, Springer, Berlin, 1980.

\bibitem{Schultz-2}
R. Schultz, Group actions on hypertoral manifolds. II. {\em J. Reine Angew. Math.} {\bf 325} (1981), 75--86.

\bibitem{Serre}
J.-P. Serre,
Bounds for the orders of the finite subgroups of $G(k)$,
{\em Group representation theory}, 405--450, EPFL Press, Lausanne, 2007.

\bibitem{tD}
T. tom Dieck,
Transformation groups.
De Gruyter Studies in Mathematics, 8. Walter de Gruyter \& Co., Berlin, 1987. x+312 pp.

\bibitem{vL}
W. van Limbeek,
Symmetry gaps in Riemannian geometry and minimal orbifolds,
{\em J. Differential Geom.} {\bf 105} (2017), no. 3, 487--517.

\bibitem{Wall}
C.T.C. Wall,
Surgery on compact manifolds.
Second edition. Edited and with a foreword by A. A. Ranicki. Mathematical Surveys and Monographs, 69. American Mathematical Society, Providence, RI, 1999.

\bibitem{Xu}
J. Xu, Finite $p$-groups of birational automorphisms and characterizations of rational
varieties, {\em preprint} {\tt arXiv:1809.09506}.

\bibitem{Yau}
S.T. Yau, Remarks on the group of isometries of a Riemannian manifold,
{\em Topology} {\bf 16} (1977), no. 3, 239--247.

\bibitem{Ye}
S. Ye, Symmetries of flat manifolds, Jordan property and the general Zimmer program,
J. Lond. Math. Soc. (2) {\bf 100} (2019), no. 3, 1065--1080.


\bibitem{Zimmermann2014} B.P. Zimmermann, On Jordan
    type bounds for finite
    groups acting on compact $3$-manifolds,
    {\em Arch. Math.} {\bf 103} (2014), 195--200.
\end{thebibliography}
\end{document}